\documentclass{amsart}
\title{The Structure of the Grothendieck Rings of Wreath Product Deligne Categories and their Generalisations}
\author{Christopher Ryba}
\address{Department of Mathematics, Massachusetts Institute of Technology, Cambridge, MA 02139, USA}
\email{ryba@mit.edu}
\date{\today}
\usepackage{hyperref}
\usepackage{fullpage}
\usepackage{amsmath}
\usepackage{amsthm}
\usepackage{amsfonts}
\usepackage{amssymb}
\usepackage{ytableau}
\usepackage{subfig}
\usepackage{graphicx, caption, float}
\usepackage{verbatim}
\usepackage{todonotes}
\usepackage{ytableau}
\usepackage{esvect}
\begin{document}
\maketitle

\newtheorem{theorem}{Theorem}[section]
\newtheorem{lemma}[theorem]{Lemma}
\newtheorem{proposition}[theorem]{Proposition}
\newtheorem{corollary}[theorem]{Corollary}
\newtheorem{definition}[theorem]{Definition}
\newtheorem{remark}[theorem]{Remark}
\newtheorem{example}[theorem]{Example}

\newcommand\numberthis{\addtocounter{equation}{1}\tag{\theequation}}

\newcommand{\Res}[2] {
  \textnormal{Res}_{#1}^{#2} 
}

\newcommand{\Ind}[2] {
  \textnormal{Ind}_{#1}^{#2} 
}

\newcommand{\Coind}[2] {
  \textnormal{Coind}_{#1}^{#2} 
}

\newcommand{\Wreath}[1] {
  \mathcal{W}_{#1}(\mathcal{C})
}

\newcommand{\Groth}[0] {
  \mathcal{G}(\mathcal{C})
}

\newcommand{\End}[0] {
  \textnormal{End}
}

\newcommand{\Mat}[0] {
  \textnormal{Mat}
}

\newcommand{\Id}[0] {
  \textnormal{Id}
}

\newcommand{\hopf}[0] {
  \mathcal{R}
}

\begin{abstract}
Given a tensor category $\mathcal{C}$ over an algebraically closed field of characteristic zero, we may form the wreath product category $\mathcal{W}_n(\mathcal{C})$. It was shown in \cite{Ryba} that the Grothendieck rings of these wreath product categories stabilise in some sense as $n \to \infty$. The resulting ``limit'' ring, $\mathcal{G}_\infty^{\mathbb{Z}}(\mathcal{C})$, is isomorphic to the Grothendieck ring of the wreath product Deligne category $S_t(\mathcal{C})$ as defined by \cite{Mori}. This ring only depends on the Grothendieck ring $\mathcal{G}(\mathcal{C})$. Given a ring $R$ which is free as a $\mathbb{Z}$-module, we construct a ring $\mathcal{G}_\infty^{\mathbb{Z}}(R)$ which specialises to $\mathcal{G}_\infty^{\mathbb{Z}}(\mathcal{C})$ when $R = \mathcal{G}(\mathcal{C})$. We give a description of $\mathcal{G}_\infty^{\mathbb{Z}}(R)$ using generators very similar to the basic hooks of \cite{Nate}. We also show that $\mathcal{G}_\infty^{\mathbb{Z}}(R)$ is a $\lambda$-ring wherever $R$ is, and that $\mathcal{G}_\infty^{\mathbb{Z}}(R)$ is (unconditionally) a Hopf algebra. Finally we show that $\mathcal{G}_\infty^{\mathbb{Z}}(R)$ is isomorphic to the Hopf algebra of distributions on the formal neighbourhood of the identity in $(W\otimes_{\mathbb{Z}} R)^\times$, where $W$ is the ring of Big Witt Vectors.

\end{abstract}

\tableofcontents

\section{Introduction}
\noindent
Let $R$ be a ring which is free as a $\mathbb{Z}$-module. We define a ring $\mathcal{G}_\infty^{\mathbb{Z}}(R)$ which is an integral form for an infinite tensor product of $U(\mathbb{Q} \otimes_{\mathbb{Z}} R)$ (the universal enveloping algebra of $\mathbb{Q} \otimes_\mathbb{Z} R$, viewed as a Lie algebra). Given a tensor category (or even a ring category in the sense of Definition 4.2.3 of \cite{EGNO}) over an algebraically closed field of characteristic zero, we may form $S_t(\mathcal{C})$, the wreath product Deligne category (as in \cite{Mori}). Our construction recovers the Grothendieck ring $\mathcal{G}(S_t(\mathcal{C}))$ when $R$ is the Grothendieck ring $\mathcal{G}(\mathcal{C})$.
\newline \newline \noindent
It was shown in \cite{Nate} that $\mathcal{G}(S_t(\mathcal{C}))$ is a filtered ring with associated graded ring isomorphic to a free polynomial algebra in certain elements called ``basic hooks'', indexed by $\mathbb{Z}_{>0} \times I(\mathcal{C})$, where $I(\mathcal{C})$ is the set of isomorphism classes of simple objects in $\mathcal{C}$. We show an analogous result for our case, indexed by $\mathbb{Z}_{>0} \times I$, where $I$ is a $\mathbb{Z}$-basis of $R$. If $U \in R$ we define the elements $e_n(U)$ by giving an expression for the generating function of the $e_n(U)$ (see Definition \ref{E_function}):
\[
E_U(t) = \sum_{n \geq 0} e_n(U)t^n
\]
where $e_0(U) = 1$. This allows us to state one family of relations between the elements $e_n(U)$; for $U, V \in R$,
\[
E_U(u)E_{VU}(-uv)^{-1}E_V(v) = E_V(v)E_{UV}(-uv)^{-1}E_U(u).
\]
We also discuss how to express $e_n(V)$ in terms of $e_n(U)$ where $U \in I$ (where $I$ is any fixed basis of $R$). For this we define an auxiliary generating function (with rational coefficients):
\[
F_U(t) = -\sum_{r \geq 1} \frac{\mu(r)}{r} \log(E_{U^r}(-t^r)),
\]
We then have the following relations (for any $U, V \in R$):
\[
F_{U+V}(t) = F_U(t) + F_V(t)
\]
This equation of generating functions may be rewritten to have integral coefficients (although we do not do this explicitly), leading to a description of $\mathcal{G}_\infty^{\mathbb{Z}}(R)$.
\newline \newline \noindent
We show that $\mathcal{G}_\infty^{\mathbb{Z}}(R)$ has a $\lambda$-ring structure whenever $\mathcal{G}(R)$ does. In the case where $R$ is the Grothendieck ring of a symmetric tensor category, so that $R$ inherits a $\lambda$-ring structure, the $\lambda$-ring structure on $\mathcal{G}_\infty^{\mathbb{Z}}(R) = \mathcal{G}(S_t(\mathcal{C}))$ is the same as the one induced by the symmetric tensor structure on $S_t(\mathcal{C})$.
\newline \newline \noindent
This algebra $\mathcal{G}_\infty^{\mathbb{Z}}(R)$ has a Hopf algebra structure (when $R = \mathcal{G}(\mathcal{C})$, the comultiplication is induced by certain functors between Deligne categories). The comultiplication $\Delta$, counit $\varepsilon$, and antipode $S$, are defined as follows:
\begin{eqnarray*}
\Delta(E_U(t)) &=& E_U(t) \otimes E_U(t) \\
\varepsilon(E_U(t)) &=& 1 \\
S(E_U(t)) &=& E_U(t)^{-1}
\end{eqnarray*}
In particular, the generating function $E_U(t)$ is grouplike.
\newline \newline \noindent
Finally, we prove that $\mathcal{G}_\infty^{\mathbb{Z}}(R)$ is isomorphic to the Hopf algebra of distributions on the formal neighbourhood the identity of $(W \otimes_{\mathbb{Z}} R)^\times$, where $W$ is the ring of Big Witt Vectors.
\newline \newline \noindent
The structure of the paper is as follows. In Section 2 we introduce the main algebraic tools relevant to us; the ring of symmetric functions, as well as some discussion of wreath products of tensor categories. The structure of $\mathbb{Q} \otimes_\mathbb{Z} \mathcal{G}(S_t(\mathcal{C}))$ was established in \cite{Ryba} and is discussed in Section 3. We define the main object of the paper, $\mathcal{G}_\infty^{\mathbb{Z}}(R)$ in Section 4, along with the generators $e_n(U)$. In Sections 5 and 6, we prove certain relations satisfied by $e_n(U)$. Section 7 shows how a $\lambda$-ring structure on $R$ passes to a $\lambda$-ring structure on $\mathcal{G}_\infty^{\mathbb{Z}}(R)$, while Section 8 constructs a Hopf algebra structure on $\mathcal{G}_\infty^{\mathbb{Z}}(R)$ (note that $R$ is not required to be a Hopf algebra for this). Finally in Section 9 it is shown that, as a Hopf algebra, $\mathcal{G}_\infty^{\mathbb{Z}}(R)$ admits a description as a certain algebra of distributions.

\subsection{Acknowledgements}
The author would like to thank Pavel Etingof and Hood Chatham for useful conversations.

\section{Background}
\subsection{Partitions and Symmetric Functions}
Throughout this paper we will make use of symmetric function combinatorics. We now introduce the notation we will use; we refer to \cite{Macdonald} for background on the topic.
\newline \newline \noindent
Recall that a partition $\lambda$ of $n$ is a sequence of nonnegative integers, $\lambda = (\lambda_1, \lambda_2, \ldots, \lambda_l)$ such that $\lambda_1 \geq \lambda_2 \geq \cdots \geq \lambda_l$, and $\lambda_1 + \lambda_2 + \cdots + \lambda_l = n$; we say $n$ is the size of $\lambda$, denoted $|\lambda|$. The $\lambda_i$ are called the parts of $\lambda$. We consider two partitions to be the same if one is obtained from the other by appending zeros to the end. An alternative way of specifying a partition $\lambda$ is to give the numbers $m_i = |\{ r \mid \lambda_r = i\}|$ which count the number of parts of $\lambda$ of size $i$; in this case we write $\lambda = (1^{m_1} 2^{m_2} \cdots )$ (thus $|\lambda| = \lambda_1 + \lambda_2 + \cdots = 1m_1 + 2m_2 + 3m_3 + \cdots$). When it is unclear which partition we are considering, we write $m_i(\lambda)$ for the number of parts of size $i$ in the partition $\lambda$. The length of $\lambda$, denoted $l(\lambda)$, is the number of nonzero parts of $\lambda$; $l(\lambda) = m_1 + m_2 + \cdots$. We will make use of the quantities $z_\lambda = \prod_{i} m_i! i^{m_i}$ and $\varepsilon_\lambda = (-1)^{\sum_i (i-1)m_i}$.
\newline \newline \noindent
We write $\lambda \vdash n$ to mean that $\lambda$ is a partition of $n$, and $\mathcal{P} = \{\lambda \mid \lambda \vdash n, n \in \mathbb{N}_{\geq 0} \}$ is the set of all partitions.

\begin{definition} \label{gen_part_one}
Let $\lambda = (\lambda_1, \lambda_2, \ldots, \lambda_l)$ be a partition, and assume $n$ is a natural number such that $n \geq |\lambda| + \lambda_1$. We write $\lambda[n]$ for the partition $(n-|\lambda|, \lambda_1, \lambda_2, \ldots, \lambda_l)$.
\end{definition}
\noindent
Note that the inequality on $n$ guarantees that this sequence is weakly decreasing.
\noindent
Thus $\lambda[n]$ is the partition of $n$ obtained by adding a first part to $\lambda$ (or top row when the partition is depicted as a Young diagram in English notation) of the appropriate size.
\newline \newline \noindent
The main algebraic tool we use is the ring of symmetric functions, denoted $\Lambda$. It is defined via a graded inverse limit of rings. Consider the rings of invariants $R_n = \mathbb{Z}[x_1,x_2,\cdots,x_n]^{S_n}$, with the symmetric group $S_n$ acting by permutation of variables. There are homomorphisms $R_{n+1} \to R_n$ defined by setting $x_{n+1} = 0$, and the graded inverse limit defined by these is $\Lambda$. We refer the reader to \cite{Macdonald} for details.
\newline \newline \noindent
As a commutative $\mathbb{Z}$-algebra, $\Lambda$ is freely generated by the elementary symmetric functions $e_i$, as well as the complete symmetric functions $h_i$. In particular $\Lambda = \mathbb{Z}[e_1, e_2, \ldots] = \mathbb{Z}[h_1, h_2, \ldots]$ (by convention $h_0=e_0=1$). There is an automorphism $\omega$ of $\Lambda$ defined by $\omega(e_i) = h_i$. It turns out that $\omega$ is an involution: $\omega(h_i) = e_i$.
\newline \newline \noindent
Another important family of elements of $\Lambda$ are the power-sum symmetric functions $p_n$ which do not generate $\Lambda$ over $\mathbb{Z}$, however, $\mathbb{Q} \otimes_\mathbb{Z} \Lambda = \mathbb{Q}[p_1, p_2, \ldots]$. The relations between the $e_i$, $h_i$, and $p_i$ are encapsulated in relations between the following formal power series: let $E(t) = \sum_{n=0}^{\infty} e_n t^n$, $H(t) = \sum_{n=0}^{\infty} h_n t^n$ (here $e_0=h_0=1$), and $P(t) = \sum_{n=0}^{\infty} p_{n+1}t^n$ be the generating functions of the three families of symmetric functions we have introduced. Then $H(t)E(-t) = 1$, and $\frac{E^{\prime}(t)}{E(t)} = P(-t)$.
\newline \newline \noindent
Given a partition $\lambda = (\lambda_1, \lambda_2, \ldots, \lambda_l)$, let $e_\lambda = e_{\lambda_1} e_{\lambda_2} \cdots e_{\lambda_l}$, and also let $h_\lambda$ and $p_\lambda$ be defined analogously. The $e_\lambda$ and $h_\lambda$ (for $\lambda \in \mathcal{P}$) each form bases of $\Lambda$, whilst the $p_\lambda$ form a basis of $\mathbb{Q} \otimes \Lambda$. Finally, we let $s_\lambda$ denote Schur functions (indexed by $\lambda \in \mathcal{P}$); Schur functions are a $\mathbb{Z}$-basis of $\Lambda$.
\newline \newline
\noindent
The Schur functions are related to the other bases via the representation theory of the symmetric groups. The simple $\mathbb{Q}S_n$-modules are indexed by partitions $\lambda \vdash n$. They are called Specht modules and denoted $\mathcal{S}^{\lambda}$. Note that the cycle type of a permutation describes its conjugacy class in $S_n$, and these are also parametrised by partitions of $n$; let $\chi_{\mu}^{\lambda}$ denote the value of the character of $\mathcal{S}^\lambda$ on an element of cycle type $\mu$. It is an important fact that these describe the change-of-basis matrix between the $s_\lambda$ and the $p_\mu$:
\[
s_\lambda = \sum_{\mu \vdash |\lambda|} \frac{\chi_{\mu}^{\lambda}p_\mu}{z_\mu}, \hspace{10mm} p_\mu = \sum_{\lambda \vdash |\mu|} \chi_{\mu}^{\lambda} s_\lambda.
\]
We have two special cases coming from the equations $h_n = s_{(n)}$ and $e_n = s_{(1^n)}$ (where $\chi_{\mu}^{(n)} = 1$ and $\chi_{\mu}^{(1^n)} = \varepsilon_\mu$). Namely, we have:
\[
h_n = \sum_{\lambda \vdash n} \frac{p_\lambda}{z_\lambda}, \hspace{20mm}
e_n = \sum_{\lambda \vdash n} \frac{\varepsilon_\lambda p_\lambda}{z_\lambda}.
\]
There is a bilinear form $\langle -, - \rangle$ on $\Lambda$ for which the Schur functions are orthonormal. It satisfies $\langle p_\lambda, p_\mu \rangle = \delta_{\lambda , \mu} z_\lambda$, where $\delta_{\lambda, \mu}$ is the Kronecker delta. 
\newline \newline \noindent
An alternative description of $\Lambda \otimes \Lambda$ is given by functions that are symmetric in two sets of variables separately. We write $f(\mathbf{x})$ to indicate that $f$ is a symmetric function in the variables $\{ x_i\}$ (suppressing the index set of the variables), or $f(\mathbf{x},\mathbf{y})$ to mean that $f$ is a symmetric function in the variables $\{ x_i\} \cup \{ y_j\}$. We also write $f(\mathbf{x}\mathbf{y})$ when the variables are $\{ x_i y_j \}$ (for example, $p_n(\mathbf{x}, \mathbf{y}) = p_n(\mathbf{x}) + p_n(\mathbf{y})$ and $p_n(\mathbf{x}\mathbf{y}) = p_n(\mathbf{x})p_n(\mathbf{y})$). We have the Cauchy identity:
\[
\prod_{i,j} \frac{1}{1-x_i y_j} = \sum_{\lambda \in \mathcal{P}} s_\lambda(\mathbf{x})s_\lambda(\mathbf{y})
= \sum_{\mu \in \mathcal{P}} \frac{p_\mu(\mathbf{x}) p_\mu(\mathbf{y})}{z_\mu} = \exp \left( \sum_{l \geq 1} \frac{p_l(\mathbf{x})p_l(\mathbf{y})}{l}\right).
\]
We also have that $s_\lambda(\mathbf{x}\mathbf{y}) = \sum_{\mu, \nu \in \mathcal{P}} k_{\mu,\nu}^{\lambda}s_\mu(\mathbf{x}) s_\nu(\mathbf{y})$, where $k_{\mu, \nu}^{\lambda}$ are the Kronecker coefficients, defined when $|\lambda| = |\mu| = |\nu|$ as tensor product multiplicities for Specht modules (we take $k_{\mu, \nu}^\lambda = 0$ if $\mu, \nu, \lambda$ are not all of the same size):
\[
\mathcal{S}^{\mu} \otimes \mathcal{S}^\nu = \bigoplus_{\lambda} \left( \mathcal{S}^\lambda \right)^{\oplus k_{\mu,\nu}^{\lambda}}.
\]
We also have $s_{\lambda}(\mathbf{x},\mathbf{y}) = \sum_{\mu, \nu} c_{\mu, \nu}^\lambda s_\mu(\mathbf{x}) s_\nu(\mathbf{y})$, where $c_{\mu, \nu}^{\lambda}$ are the Littlewood-Richardson coefficients. These coefficients also satisfy $s_\mu s_\nu = \sum_{\lambda} c_{\mu, \nu}^{\lambda}s_\lambda$. We take the Littlewood-Richardson coefficient $c_{\mu, \nu}^{\lambda}$ to be zero if $|\mu| + |\nu| \neq |\lambda|$. We will consider symmetric functions with many variable sets, some of which may be repeated. To indicate that a certain variable set occurs multiple times, we write $f(\mathbf{x}, \mathbf{x}, \ldots, \mathbf{x}) = f(\mathbf{x}^{\oplus n})$ (where $n$ is the number of occurrences of $\mathbf{x}$ in the left hand side). More generally, if we have a family of variable sets $\mathbf{x}^{(i)}$, then we write $f(\mathbf{x}^{(i_1)}, \mathbf{x}^{(i_2)}, \ldots) = f(\bigoplus_i \mathbf{x}^{(i)})$.
\newline \newline \noindent
There is a Hopf algebra structure on $\Lambda$. The comultiplication, which we denote $\Delta^{(+)}: \Lambda \to \Lambda \otimes \Lambda$, is given by $\Delta^{(+)}(f) = f(\mathbf{x}, \mathbf{y})$ (here we consider $\Lambda \otimes \Lambda$ as symmetric functions in two sets of variables). This makes the power-sum symmetric functions primitive. When we consider the Big Witt Vectors, we will need the Kronecker comultiplication; $\Delta^{(\times )}(f) = f(\mathbf{xy})$ (it is part of a Hopf algebra structure with a different multiplication which we do not consider in this paper).
\newline \newline \noindent
We will consider a tensor product of copies of $\Lambda$ indexed by a set $I$:
\[
\bigotimes_{U \in I} \Lambda^{(U)}
\]
Similarly to the case of $\Lambda \otimes \Lambda$, this is the space of symmetric functions in many sets of variables, which we may denote $\mathbf{x}_U$ for $U \in I$. To indicate that a symmetric function $f$ belongs to $\Lambda^{(U)}$ (or its canonical inclusion into $\bigotimes_{U \in I} \Lambda^{(U)}$), we write either $f(\mathbf{x}_U)$, or $f^{(U)}$ when we are not concerned with the variable sets themselves.

\subsection{Tensor Categories and Wreath Products}
The main cases of interest and motivating examples for our results will involve tensor categories. Let $k$ be an algebraically closed field of characteristic zero and $\mathcal{C}$ an artinian tensor category over $k$. All our results will hold when $\mathcal{C}$ is an artinian ring category over $k$ in the sense of Section 4.2 of \cite{EGNO}. However, a reader who is not interested in this degree of generality may take $\mathcal{C}$ to be the category of finite-dimensional modules over a Hopf algebra $\mathcal{R}$ over $k$. We write $I(\mathcal{C})$ for the set of isomorphism classes of simple objects of $\mathcal{C}$, and $\mathbf{1}$ for the unit object. Throughout, $\mathcal{G}(\mathcal{D})$ will indicate the Grothendieck group or ring of a category $\mathcal{D}$.
\newline \newline \noindent
Wreath products underlie the main object of this paper. In the case where $\mathcal{C}$ is $\mathcal{R}-mod$, we may consider the wreath product $\mathcal{R} \wr S_n$; as a vector space, this is $\mathcal{R}^{\otimes n} \otimes kS_n$. The multiplication on $\mathcal{R} \wr S_n$ is determined by requiring the maps $\varphi: \mathcal{R}^{\otimes n} \to \mathcal{R}^{\otimes n} \otimes kS_n: x \mapsto x \otimes 1_{S_n}$ and $\psi: kS_n \to \mathcal{R}^{\otimes n} \otimes kS_n: y \mapsto 1_{\mathcal{R}} \otimes y$ to be algebra homomorphisms, together with the following commutation relation between the Hopf algebra and symmetric group parts. If $r_1, r_2, \ldots, r_n \in \mathcal{R}$, and $\sigma \in S_n$,
\[
\varphi(r_1 \otimes r_2 \otimes \cdots \otimes r_n) \psi(\sigma) = \psi(\sigma) \varphi(r_{\sigma(1)} \otimes r_{\sigma(2)} \otimes \cdots \otimes r_{\sigma(n)}).
\]
This is again a Hopf algebra, with the comultiplication and antipode determined by requiring $\varphi$ and $\psi$ to be homomorphisms of Hopf algebras. In the case of more general $\mathcal{C}$, we have the following definition. Write $\boxtimes$ for the Deligne product of categories (see Section 1.11 of \cite{EGNO}). We may take the $n$-fold iterated product of $\mathcal{C}$ with itself, $\mathcal{C}^{\boxtimes n}$. This comes with an action of the symmetric group $S_n$ by permutation of the factors. The equivariantisation with respect to this action is the wreath product category we are concerned with; we write $\mathcal{W}_n(\mathcal{C}) = (\mathcal{C}^{\boxtimes n})^{S_n}$.
\newline \newline \noindent
The simple objects of the wreath product category $\mathcal{W}_n(\mathcal{C})$ are indexed by the following set:
\begin{definition}
Let $\mathcal{P}_n^{\mathcal{C}}$ denote partition-valued functions on $I(\mathcal{C})$ with total size $n$:
\[
\mathcal{P}_n^{\mathcal{C}} = \{\vv{\lambda}: I(\mathcal{C}) \to \mathcal{P} \mid \sum_{U \in I(\mathcal{C})} |\vv{\lambda}(U)| = n \}
\]
\end{definition}
\noindent
We use the vector notation $\vv{\lambda}$ to denote a partition-valued function on $I(\mathcal{C})$, and write $|\vv{\lambda}| = \sum_{U \in I(\mathcal{C})} |\vv{\lambda}(U)|$ for the total size of $\vv{\lambda}$. We will also write $S_{\vv{\lambda}}$ for the Young subgroup $\prod_{U \in I(\mathcal{C}} S_{|\vv{\lambda}(U)|}$ of $S_{|\vv{\lambda}|}$. We now explain how to construct the simple object corresponding to $\vv{\lambda}$.
\newline \newline \noindent
The object corresponding to a multipartition $\vv{\lambda}$ can be explicitly constructed using induction functors as described in \cite{Mori}; we briefly summarise this theory. Suppose that a group $G$ acts on the category $\mathcal{C}$. Then any subgroup of $G$ also acts on $\mathcal{C}$, via restriction. As in Section 3.2 of \cite{Mori}, if $\mathcal{D}$ is an additive category, then given any finite-index subgroup $H$ of $G$, we have a forgetful functor $\Res{H}{G}: \mathcal{D}^G \to \mathcal{D}^H$ (the superscript group indicates the equivariantisation of $\mathcal{D}$ with respect to that group). Further, there is an induction functor $\Ind{H}{G}: \mathcal{D}^H \to \mathcal{D}^G$ which is two-sided adjoint to $\Res{H}{G}$. The induction functor may be written as a sum over coset representatives of $H$ in $G$ as follows:
\[
\Ind{H}{G}(M) = \bigoplus_{g \in G/H} gM
\]
The action of $G$ is completely analogous to the case of induction of representations of finite groups, where $M$ would be a representation of $H$.
\newline \newline \noindent
Equipped with this notion, we consider $U^{\boxtimes |\vv{\lambda}(U)|} \otimes \mathcal{S}^{\vv{\lambda}(U)}$; this defines an object of $\mathcal{C}^{\boxtimes |\vv{\lambda}(U)|}$. Then, $S_{|\vv{\lambda}(U)|}$ acts by permuting the tensor factors of $U^{\boxtimes |\vv{\lambda}(U)|}$, whilst also acting on $\mathcal{S}^{\vv{\lambda}(U)}$. Hence, we have an object of $\mathcal{W}_{|\vv{\lambda}(U)|}(\mathcal{C})$. We may take the (Deligne) external product of these objects (for $U \in I(\mathcal{C})$) to get:
\[
\boxtimes_{U \in I(\mathcal{C})} \left(U^{\boxtimes |\vv{\lambda}(U)|} \otimes \mathcal{S}^{\vv{\lambda}(U)}\right)
\]
which is an object of the category
\[
\boxtimes_{U \in I(\mathcal{C})} \left( \mathcal{W}_{|\vv{\lambda}(U)|}(\mathcal{C}) \right) \cong (\mathcal{C}^{\boxtimes |\vv{\lambda}|})^{S_{\vv{\lambda}}}.
\]
Upon applying the induction functor $\Ind{S_{\vv{\lambda}}}{S_{|\vv{\lambda}|}}$, we obtain the simple object corresponding to the multipartition $\vv{\lambda}$. (In Section 5 of \cite{Mori} an analogous statement is proved for indecomposable objects when $\mathcal{C}$ is an additive category, but our case is not materially different.)
\newline \newline \noindent
Since the simple objects are induced from Young subgroups of symmetric groups, the tensor structure of the categories $\mathcal{W}_n(\mathcal{C})$ can be studied using Mackey theory. In \cite{Ryba} this approach was used to demonstrate certain stability properties. To state them, we need the following definition (analogous to Definition \ref{gen_part_one}).

\begin{definition}
If $\vv{\lambda} \in \mathcal{P}_m^{\mathcal{C}}$, for $n \geq m + \vv{\lambda}(\mathbf{1})_1$, we define $\vv{\lambda}[n] \in \mathcal{P}_n^{\mathcal{C}}$ to be multpartition constructed in the following way. For $U \in I(\mathcal{C})$ different from $\mathbf{1}$, $\vv{\lambda}[n](U) = \vv{\lambda}(U)$ whilst $\vv{\lambda}[n](\mathbf{1}) = \vv{\lambda}(\mathbf{1})[n]$.
\end{definition}
\noindent
Write $R_{\vv{\lambda}}$ for the simple object of $\mathcal{W}_n(\mathcal{C})$ indexed by the multipartition $\vv{\lambda} \in \mathcal{P}_n^{\mathcal{C}}$. In the Grothendieck ring $\mathcal{G}(\mathcal{W}_n(\mathcal{C}))$, consider the multiplication induced by the tensor product. Here, $\vv{\lambda}, \vv{\mu}, \vv{\nu}$ are fixed multipartitions of any sizes, square brackets indicate taking the image in the Grothendieck group, and we take $n$ sufficiently large:
\[
[R_{\vv{\mu}[n]}][R_{\vv{\nu}[n]}] = [R_{\vv{\mu}[n]} \otimes R_{\vv{\nu}[n]}] = \sum_{\vv{\lambda}} c_{\vv{\mu}, \vv{\nu}}^{\vv{\lambda}}(n) [R_{\vv{\lambda}[n]}]
\]
The sum is taken only over those $\vv{\lambda}$ such that $\vv{\lambda}[n]$ is defined. This equation serves to define the structure constants $c_{\vv{\mu}, \vv{\nu}}^{\vv{\lambda}}(n)$ for $n$ sufficiently large. Theorem 7.10 of \cite{Ryba} states the following: $\lim_{n \to \infty} c_{\vv{\mu}, \vv{\nu}}^{\vv{\lambda}}(n)$ exists, and is finite. Further, for fixed $\vv{\mu}, \vv{\nu}$, it is nonzero for only finitely many $\vv{\lambda}$. With some further work (Theorems 7.14 and 7.15 of \cite{Ryba}), it follows that we may define an associative algebra $\mathcal{G}_\infty^{\mathbb{Z}}(\mathcal{C})$ which is in some sense ``the $n \to \infty$ limit of the Grothendieck ring of $\mathcal{W}_n(\mathcal{C})$''. 
\begin{definition}
Let $\mathcal{G}_\infty^{\mathbb{Z}}(\mathcal{C})$ have $\mathbb{Z}$-basis $\{X_{\vv{\lambda}}\}$ indexed by all multipartitions of finite size $\{\vv{\lambda} \mid \vv{\lambda} \in \mathcal{P}_n^{\mathcal{C}}, n \in \mathbb{N} \}$, and multiplication given by:
\[
X_{\vv{\mu}} X_{\vv{\nu}} = \sum_{\vv{\lambda}} \left( \lim_{n \to \infty} c_{\vv{\mu}, \vv{\nu}}^{\vv{\lambda}}(n) \right) X_{\vv{\lambda}}
\]
We write $\mathcal{G}_\infty^{\mathbb{Q}}(\mathcal{C}) = \mathbb{Q} \otimes_{\mathbb{Z}} \mathcal{G}_\infty^{\mathbb{Z}}(\mathcal{C})$, when we take rational coefficients.
\end{definition}
\noindent
By Theorem 7.14 of \cite{Ryba} it follows that $\mathcal{G}_\infty^{\mathbb{Z}}(\mathcal{C})$ is the Grothendieck ring of the wreath-product Deligne category $S_t(\mathcal{C})$ as introduced by \cite{Mori}, with $X_{\vv{\lambda}}$ corresponding to the simple objects (for generic $t$) of $S_t(\mathcal{C})$. It was proved in \cite{Nate} that the Grothendieck ring of $S_t(\mathcal{C})$ has a filtration (``$|\lambda|$-filtration''), and a generating set was given (called ``basic hooks''). In \cite{Ryba}, the algebraic structure of $\mathcal{G}_\infty^{\mathbb{Q}}(\mathcal{C})$ was established. Theorem 8.8 of \cite{Ryba} asserts that $\mathcal{G}_\infty^{\mathbb{Q}}(\mathcal{C})$ is isomorphic to an infinite tensor product of universal enveloping algebras of the Lie algebra obtained by taking the Grothendieck ring of $\mathcal{C}$ with rational coefficients. That is:
\[
\mathcal{G}_\infty^{\mathbb{Q}}(\mathcal{C}) = \bigotimes_{i=1}^{\infty} U(\mathbb{Q} \otimes_\mathbb{Z} \mathcal{G}(\mathcal{C})_i)
\]
Here, the subscript $i$ in $\mathcal{G}(\mathcal{C})_i$ means that the elements of this copy of $\mathcal{G}(\mathcal{C})$ lie in filtration degree $i$. The infinite tensor product is spanned by pure tensors $a_1 \otimes a_2 \otimes \cdots$ such that all but finitely many $a_i = 1$.
\newline \newline
\noindent
The purpose of this paper is to describe the integral version of the algebra, $\mathcal{G}_\infty^{\mathbb{Z}}(\mathcal{C})$, and to elucidate some other properties, such as a $\lambda$-ring structure (when there is a $\lambda$-ring structure on the Grothendieck ring $\mathcal{G}(\mathcal{C})$), a Hopf-algebra structure, and a realisation as a certain algebra of distributions.


\section{Structure of the Rational Limiting Grothendieck Ring of Wreath Product Categories}
\noindent
We now discuss the details of the structure of the rational limiting Grothendieck ring $\mathcal{G}_\infty^{\mathbb{Q}}(\mathcal{C})$. The main tool is the family of elements $T_n(U)$ (constructed in Section 8 of \cite{Ryba}), where $n \in \mathbb{Z}_{\geq 0}$ and $U$ can be any element of the Grothendieck ring $\mathcal{G}(\mathcal{C})$. The elements $T_n(U)$ are first defined for the case where $U$ is the image of an object of $\mathcal{C}$ in the Grothendieck ring. It is then shown that if $N$ is a subobject of $M$ (so that $0 \to N \to M \to M/N \to 0$ is a short exact sequence), then $T_n(M) = T_n(N) + T_n(M/N)$. This means that $T_n(U)$ is linear in $U$ for $U$ that are the images of objects of $\mathcal{C}$ in the Grothendieck ring. It also means that $T_n(U)$ only depends on the composition factors of $U$ (with multiplicity), so we will not worry about distinguishing $U$ from its image in the Grothendieck group. This linearity property justifies defining $T_n(U)$ for arbitrary $U \in \mathcal{G}(\mathcal{C})$ by writing $U = U_1 - U_2$ where $U_1$ and $U_2$ are the images of objects of $\mathcal{C}$, and letting $T_n(U) = T_n(U_1) - T_n(U_2)$. It is then shown that 
\[
T_n(U)T_n(V) - T_n(V)T_n(U) = T_n(UV) -T_n(VU),
\]
where the multiplication of $U$ and $V$ takes place in $\mathcal{G}(\mathcal{C})$. Also, $T_m(U)$ and $T_n(V)$ commute when $m \neq n$. Although the $T_n(U)$ are not elements of $\mathcal{G}_\infty^{\mathbb{Z}}(\mathcal{C})$ unless $n=1$, they do generate $\mathcal{G}_\infty^{\mathbb{Q}}(\mathcal{C})$.
\newline \newline \noindent
The way to express an element $X_{\vv{\lambda}}$ in terms of the $T_n(U)$ is encoded in a certain identity of generating functions. It is given by Theorem 9.10 of \cite{Ryba}:
\begin{theorem} \label{old_gen_fun}
Write $\Lambda_{\mathbb{Q}}^{(U)}$ for a copy of the ring of symmetric functions with rational coefficients, whose variables we associate with $U \in I(\mathcal{C})$. We work in the completed tensor product $\left(\bigotimes_{U \in I(\mathcal{C})} \Lambda_{\mathbb{Q}}^{(U)}\right) \hat{\otimes} \mathcal{G}(\mathcal{C})$. If $f$ is a symmetric function, we write $f^{(U)}$ to denote $f$ considered as an element of $\Lambda_{\mathbb{Q}}^{(U)}$. If $c_{\vv{\mu}}^{(U)}$ are constants, we define
\[
T_l\left(\sum_{\vv{\mu} \in \mathcal{P}^{\mathcal{C}}, U \in I(\mathcal{C})} c_{\vv{\mu}}^{(U)} \left( \prod_{V \in I(\mathcal{C})} p_{\vv{\mu}(V)}^{(V)} \right) U\right) = \sum_{\vv{\mu} \in \mathcal{P}, U \in I(\mathcal{C})} c_{\vv{\mu}}^{(U)} \left(\prod_{V \in I(\mathcal{C})} p_{\vv{\mu}(V)}^{(V)}\right) T_l(U)
\]
which is an element of $\left(\bigotimes_{U \in I(\mathcal{C})} \Lambda_{\mathbb{Q}}^{(U)}\right) \hat{\otimes} \mathcal{G}_\infty(\mathcal{C})$.
We have the following equality:
\[
\sum_{\vv{\lambda} \in \mathcal{P}^\mathcal{C}} \left( \prod_{U \in I(\mathcal{C})} s_{\vv{\lambda}(U)}^{(U)} \right) \otimes X_{\vv{\lambda}} =
\left(\sum_{r \geq 0} (-1)^r e_r^{(\mathbf{1})} \right)
\prod_{l=1}^{\infty}\exp \left(T_l \left(\log \left(1 + \sum_{U \in I(\mathcal{C})} p_l^{(U)}U\right) \right) \right).
\]
\end{theorem}
\noindent
This theorem encodes how to express the $X_{\vv{\lambda}}$ (quantities of representation-theoretic interest) in terms of the $T_n(U)$ (amenable to computation). For our purposes we will be interested in a slightly different generating function which encodes a slightly different basis of $\mathcal{G}_\infty^{\mathbb{Z}}(\mathcal{C})$. Consider the elements $Z_{\vv{\lambda}}$ defined by the following equation:
\[
\sum_{\vv{\lambda} \in \mathcal{P}^\mathcal{C}} \left( \prod_{U \in I(\mathcal{C})} s_{\vv{\lambda}(U)}^{(U)} \right) \otimes Z_{\vv{\lambda}} =
\prod_{l=1}^{\infty}\exp \left(T_l \left(\log \left(1 + \sum_{U} p_l^{(U)}U\right) \right) \right)
\]
Because the omitted factor $\left(\sum_{r \geq 0} (-1)^r e_r^{(\mathbf{1})} \right)$ is equal to $1$ plus higher order terms, this means that $Z_{\vv{\lambda}}$ and $X_{\vv{\lambda}}$ agree up to lower order terms in the filtration on $\mathcal{G}_\infty^{\mathbb{Z}}(\mathcal{C})$. In particular the $Z_{\vv{\lambda}}$ form a $\mathbb{Z}$-basis of $\mathcal{G}_\infty^{\mathbb{Z}}(\mathcal{C})$.

\begin{remark}
The relation between the two bases ${X}_{\vv{\lambda}}$ and $Z_{\vv{\lambda}}$ is calculated in \cite{Ryba}, where elements denoted $\lim_{m \to \infty} \Ind{S_{\vv{\lambda}} \times S_m}{S_{|\vv{\lambda}|+m}} \left(\boxtimes_{U \in I(\mathcal{C})} \left(U^{\boxtimes |\vv{\lambda}|} \otimes \mathcal{S}^{\vv{\lambda}(U)}\right) \boxtimes \left(\mathbf{1}^{\boxtimes m} \otimes \mathbf{1}_{S_m}\right)\right)$ are shown to have the same generating function as the $Z_{\vv{\lambda}}$ (so they are equal to our $Z_{\vv{\lambda}}$). Then, in Proposition 9.5 and Remark 9.6, the decomposition into $X_{\vv{\lambda}}$ is explained using the Pieri rule.
\end{remark}
\noindent
It was shown in \cite{Nate} that the associated graded algebra of $\mathcal{G}_\infty^{\mathbb{Z}}(\mathcal{C})$ is a free polynomial algebra in basic hooks, which are defined as $X_{\vv{\lambda}}$, where $\vv{\lambda}(U)$ is nonzero for a unique $U \in I(\mathcal{C})$, and for this $U$, $\vv{\lambda}(U) = (1^i)$. 

\section{General Setup}
\noindent
The algebra $\mathcal{G}_\infty^{\mathbb{Z}}(\mathcal{C})$ depends on the category $\mathcal{C}$ only via the Grothendieck ring $\mathcal{G}(\mathcal{C})$. Furthermore, the construction is still possible after replacing $\mathcal{G}(\mathcal{C})$ with a suitable ring, even if that ring is not the Grothendieck ring of a ring category. To emphasise this, we work only with rings rather than categories. To take the place of $\mathcal{G}(\mathcal{C})$, $I(\mathcal{C})$, and $\mathcal{P}^{\mathcal{C}}$, we will use:
\begin{definition}
Let $R$ be a ring which is free as a $\mathbb{Z}$-module, and let $I$ be a $\mathbb{Z}$-basis of $R$. We define $\mathcal{P}^I = \{f: I \to \mathcal{P} \mid \sum_{U \in I} |f(U)| < \infty \}$, the set of $I$-indexed multipartitions.
\end{definition}
\begin{definition}
Define
\[
\mathcal{G}_\infty^\mathbb{Q}(R) = \bigotimes_{i \geq 1} U(\mathbb{Q} \otimes R_i),
\]
where $U(\mathbb{Q}\otimes R_i)$ is a copy of the universal enveloping algebra of $\mathbb{Q} \otimes R$ (where the Lie algebra structure is inherited from the associative algebra structure). Given $r \in \mathbb{Q} \otimes R$, we write $T_i(r)$ for the image of $r$ under the canonical map $\mathbb{Q} \otimes R \to U(\mathbb{Q} \otimes R_i)$.
\end{definition}
\begin{proposition}
There is a filtration on $\mathcal{G}_\infty^{\mathbb{Q}}(R)$ uniquely determined by making $T_{i_1}(r_1)T_{i_2}(r_2)\cdots T_{i_n}(r_n)$ lie in filtration degree $i_1 + i_2 + \cdots + i_n$.
\end{proposition}
\begin{proof}
Uniqueness follows because the monomials $T_{i_1}(r_1)T_{i_2}(r_2)\cdots T_{i_n}(r_n)$ span the algebra by the PBW theorem. Existence follows from observing that the filtration is constructed in the following way. Modify the usual PBW filtration on $U(\mathbb{Q} \otimes R_i)$ by multiplying all filtration degrees by $i$, and take the induced filtration on the tensor product
\[
\bigotimes_{i \geq 1} U(\mathbb{Q} \otimes R_i).
\]
\end{proof}

\begin{proposition} \label{assoc_graded_identification}
The associated graded algebra of $\mathcal{G}_\infty^{\mathbb{Q}}(R)$ is isomorphic to $\mathbb{Q} \otimes_{\mathbb{Z}} \bigotimes_{U \in I} \Lambda^{(U)}$ where the image of $T_l(U)$ is $\frac{p_l^{(U)}}{l}$.
\end{proposition}
\begin{proof}
This follows from the fact that the associated graded algebra of $\mathcal{G}_\infty^{\mathbb{Q}}(R)$ is the free polynomial algebra generated by $T_i(U)$ for $i \in \mathbb{Z}_{> 0}$ and $U \in I$, together with the fact that $\mathbb{Q} \otimes \Lambda$ is generated by $\frac{p_i}{i}$ for $i \in \mathbb{Z}_{>0}$.
\end{proof}

\begin{definition} \label{general_Z_def}
Define $Z_{\vv{\lambda}} \in \mathcal{G}_\infty^\mathbb{Q}(R)$ (implicitly depending on $I$) by the following equality of generating functions:
\[
\sum_{\vv{\lambda} \in \mathcal{P}^\mathcal{C}} \left( \prod_{U \in I} s_{\vv{\lambda}(U)}^{(U)} \right) \otimes Z_{\vv{\lambda}} =
\prod_{l=1}^{\infty}\exp \left(T_l \left(\log \left(1 + \sum_{U \in I} p_l^{(U)}U\right) \right) \right),
\]
where, as before, 
\[
T_l\left(\sum_{\vv{\mu} \in \mathcal{P}^{\mathcal{C}}, U \in I(\mathcal{C})} c_{\vv{\mu}}^{(U)} \left( \prod_{V \in I(\mathcal{C})} p_{\vv{\mu}(V)}^{(V)} \right) U\right) = \sum_{\vv{\mu} \in \mathcal{P}, U \in I(\mathcal{C})} c_{\vv{\mu}}^{(U)} \left(\prod_{V \in I(\mathcal{C})} p_{\vv{\mu}(V)}^{(V)}\right) T_l(U).
\]
Also define $\mathcal{G}_\infty^{\mathbb{Z}}(R)$ as the $\mathbb{Z}$-span of the $Z_{\vv{\lambda}}$.
\end{definition}
\begin{theorem}
We have that $\mathcal{G}_\infty^{\mathbb{Z}}(R)$ is independent of the choice of basis $I \subset R$.
\end{theorem}
\begin{proof}
Fix two bases $I$ and $I^\prime$ of $R$, writing $Z_{\vv{\lambda}}(I)$ and $Z_{\vv{\lambda}}(I^\prime)$ for the elements $Z_{\vv{\lambda}}$ defined using the bases $I$ and $I^\prime$ respectively. Consider the symmetric functions $p_l^{(V)}$ defined by 
\[
\sum_{V \in I^\prime} p_l^{(V)}V = \sum_{U \in I} p_l^{(U)} U.
\]
If the transition matrix between $I^\prime$ and $I$ is $a_{VU}$, $V = \sum_{U \in I} a_{VU} U$. This gives
\[
p_l^{(V)} = \sum_{U \in I} a_{V U} p_l^{(U)}.
\]
This extends from power-sum symmetric functions to arbitrary symmetric functions via
\[
f^{(V)} = f\left(\bigoplus_{U \in I} \mathbf{x}_{U}^{\oplus a_{VU}}\right)
\]
where the meaning of negative multiplicities of variable sets is as in Example 23 in Section 1.3 of \cite{Macdonald} (where $f(\mathbf{x}^{\oplus 1}, \mathbf{y}^{\oplus(-1)})$ is considered, but denoted $f(\mathbf{x}/\mathbf{y})$). Because $a_{VU}$ is invertible, we have constructed an isomorphism
\[
\theta: \bigotimes_{V \in I^\prime} \Lambda^{(V)} \to \bigotimes_{U \in I} \Lambda^{(U)}
\]
which moreover satisfies
\begin{eqnarray*}
\sum_{\vv{\lambda} \in \mathcal{P}^{I^\prime}} \theta \left( \prod_{V \in I^\prime} s_{\vv{\lambda}(V)}^{(V)} \right) \otimes Z_{\vv{\lambda}}(I^\prime) &=& \prod_{l=1}^{\infty}\exp \left(T_l \left(\log \left(1 + \sum_{V \in I^\prime} \theta( p_l^{(V)})V\right) \right) \right) \\
&=& \prod_{l=1}^{\infty}\exp \left(T_l \left(\log \left(1 + \sum_{U \in I} p_l^{(U)}U\right) \right) \right) \\
&=& \sum_{\vv{\lambda} \in \mathcal{P}^{I}} \left( \prod_{U \in I} s_{\vv{\lambda}(U)}^{(U)} \right) \otimes Z_{\vv{\lambda}}(I).
\end{eqnarray*}
To recover $Z_{\vv{\lambda}}(I)$ in terms of $Z_{\vv{\lambda}}(I^\prime)$ it suffices to take the inner product of both sides with $\prod_{U \in I} s_{\vv{\lambda}(U)}^{(U)}$. In particular, this means that the $\mathbb{Z}$-linear span of $Z_{\vv{\lambda}}(I)$ is independent of the choice of basis $I$.
\end{proof}

\begin{proposition} \label{Z_graded_image}
In the associated graded algebra of $\mathcal{G}_\infty^{\mathbb{Q}}(R)$, isomorphic to $\mathbb{Q} \otimes_{\mathbb{Z}} \bigotimes_{U \in I} \Lambda^{(U)}$ (as discussed in proposition Proposition \ref{assoc_graded_identification}), the image of $Z_{\vv{\lambda}}$ is $\prod_{U \in I}s_{\vv{\lambda}(U)}^{(U)}$.
\end{proposition}
\begin{proof}
Recall that $Z_{\vv{\lambda}}$ was defined by the equation
\[
\sum_{\vv{\lambda} \in \mathcal{P}^\mathcal{C}} \left( \prod_{U \in I} s_{\vv{\lambda}(U)}^{(U)} \right) \otimes Z_{\vv{\lambda}} =
\prod_{l=1}^{\infty}\exp \left(T_l \left(\log \left(1 + \sum_{U \in I} p_l^{(U)}U\right) \right) \right).
\]
To calculate the leading coefficient of $\prod_{U \in I} s_{\vv{\lambda}(U)}^{(U)}$ as a polynomial in the $T_i(U)$, we approximate 
\[
T_l \left(\log \left(1 + \sum_{U \in I} p_l^{(U)}U\right)\right)
\]
by
\[
\sum_{U \in I} T_l(U)p_l^{(U)}.
\]
This turns the generating function
\[
\prod_{l=1}^{\infty}\exp \left(T_l \left(\log \left(1 + \sum_{U \in I} p_l^{(U)}U\right) \right) \right)
\]
into
\[
\prod_{l=1}^{\infty}\exp \left(\sum_{U \in I} p_l^{(U)}T_l(U) \right).
\]
Let the symmetric function variables be $\mathbf{x}_U$, and think of $T_l(U)$ as $\frac{p_l(\mathbf{y}_U)}{l}$ (as the $T_l(U)$ are algebraically independent in the associated graded algebra). Then, the generating function becomes
\[
\prod_{U \in I}\exp \left(\sum_{l \geq 1} \frac{p_l(\mathbf{x}_U)p_l(\mathbf{y}_U)}{l} \right).
\]
which we recognise as the product of instances of the Cauchy identity (one for each element of $I$). The coefficient of $\prod_{U \in I} s_{\vv{\lambda}(U)}(\mathbf{x})_U)$ is therefore $\prod_{U \in I} s_{\vv{\lambda}(U)}(\mathbf{y}_U)$.
\end{proof}

\begin{corollary}
The $Z_{\vv{\lambda}}$ are a basis of $\mathcal{G}_\infty^{\mathbb{Q}}(R)$.
\end{corollary}
\begin{proof}
This follows from the fact that the images of the $Z_{\vv{\lambda}}$ in the associated graded algebra (see Proposition \ref{Z_graded_image}) form a basis.
\end{proof}

\begin{definition}
Let us write
\[
\Theta_l(x) = \exp(T_l(\log(x))),
\]
where $x$ may be any series in the completed tensor product $\left( \bigotimes_{U \in I} \Lambda^{(U)}\right) \hat{\otimes} R$ whose lowest order term is $1$.
\end{definition}

\begin{remark}
With this notation we may rewrite the equation in Definition \ref{general_Z_def} as
\[
\sum_{\vv{\lambda} \in \mathcal{P}^I} \left( \prod_{U \in I} s_{\vv{\lambda}(U)}^{(U)} \right) \otimes Z_{\vv{\lambda}} =
\prod_{l=1}^{\infty}\Theta_l \left(1 + \sum_{U \in I} p_l^{(U)}U\right).
\]
\end{remark}

\begin{proposition}
We have the following identity:
\[
\Theta_l(x)\Theta_l(y) = \Theta_l(xy).
\]
\end{proposition}
\begin{proof}
We make use of the Baker-Campbell-Hausdorff formula. Let $\mbox{BCH}(A,B) = \log(\exp(A)\exp(B)) = A + B + \frac{1}{2}[A,B] + \cdots$, this is a sum of iterated commutators of $A$ and $B$ (the monomials $A$ and $B$ are thought of as a commutator iterated zero times). We have the relation $T_l([U,V]) = [T_l(U), T_l(V)]$ in the universal enveloping algebra. Let $\widehat{\otimes}$ denote the completed tensor product of graded algebras (completion of the usual tensor product with respect to the grading). Thus $T_l$ defines a Lie algebra homomorphism $\left( \bigotimes_{U \in I} \Lambda^{(U)}\right) \widehat{\otimes} R \to \left( \bigotimes_{U \in I} \Lambda^{(U)}\right) \widehat{\otimes} \mathcal{G}_\infty^{\mathbb{Q}}(R)$ (where the Lie algebra structure is inherited from the associative algebra structure, and the right hand factors have trivial gradings). We have that $T_l(BCH(-, -)) = \mbox{BCH}(T_l(-), T_l(-))$ (as $\mbox{BCH}$ is a sum of iterated commutators and $T_l(-)$ is a Lie algebra homomorphism). Now we calculate:
\begin{eqnarray*}
& &\Theta_l(x)\Theta_l(y) \\
&=&\exp(T_l(\log(x))) \exp(T_l(\log(y))) \\
&=& \exp(\mbox{BCH}(T_l(\log(x)),T_l(\log(y)))) \\
&=& \exp(T_l(\mbox{BCH}(\log(x),\log(y)))) \\
&=& 
\exp(T_l(\log(xy))) \\
&=& \Theta_l(xy)
\end{eqnarray*}

\end{proof}

\begin{theorem} \label{mult_gen_fn}
Let $N_{V,W}^U$ be the structure tensor of $R$ with respect to the basis $I$ (so that $VW = \sum_U N_{V,W}^U U$ for $U,V,W \in I$). We have the following equality of generating functions.
\begin{eqnarray*}
& &
\left(\sum_{\vv{\mu} \in \mathcal{P}^{I}} \left(\prod_{U \in I} s_{\vv{\mu}(U)}(\mathbf{x}_U) \right) \otimes Z_{\vv{\mu}} \right)
\left(\sum_{\vv{\nu} \in \mathcal{P}^{I}} \left(\prod_{U \in I} s_{\vv{\nu}(U)}(\mathbf{y}_U) \right) \otimes Z_{\vv{\nu}} \right) \\
&=&
\sum_{\vv{\lambda} \in \mathcal{P}^{I}} \left(\prod_{U \in I} s_{\vv{\lambda}(U)}(\mathbf{x}_U, \mathbf{y}_U, \bigoplus_{V, W \in I} (\mathbf{x}_V\mathbf{y}_W)^{\oplus N_{V, W}^U}) \right) \otimes Z_{\vv{\lambda}} 
\end{eqnarray*}
\end{theorem}
\begin{proof}
We use the multiplicative expression for these generating functions from Definition \ref{general_Z_def}.\begin{eqnarray*}
& &
\left(\sum_{\vv{\mu} \in \mathcal{P}^{I}} \left(\prod_{U \in I} s_{\vv{\mu}(U)}(\mathbf{x}_U) \right) \otimes Z_{\vv{\mu}} \right)
\left(\sum_{\vv{\nu} \in \mathcal{P}^{I}} \left(\prod_{U \in I} s_{\vv{\nu}(U)}(\mathbf{y}_U) \right) \otimes Z_{\vv{\nu}} \right) \\
&=&
\prod_{l=1}^{\infty}\Theta_l \left(1 + \sum_{U \in I} p_l(\mathbf{x}_U)U\right)
\prod_{l=1}^{\infty}\Theta_l \left(1 + \sum_{U \in I} p_l(\mathbf{y}_U)U\right) \\
&=&
\prod_{l=1}^{\infty}\Theta_l \left((1 + \sum_{U \in I} p_l(\mathbf{x}_U)U)(1 + \sum_{U \in I} p_l(\mathbf{y}_U)U)\right) \\
&=&
\prod_{l=1}^{\infty}\Theta_l \left(1 + \sum_{U \in I} \left(p_l(\mathbf{x}_U)+ p_l(\mathbf{y}_U) + \sum_{V, W \in I} N_{V, W}^U p_l(\mathbf{x}_V)p_l(\mathbf{y}_W)\right)U\right) \\
&=&
\prod_{l=1}^{\infty}\Theta_l \left(1 + \sum_{U} p_l\left(\mathbf{x}_U, \mathbf{y}_U,\bigoplus_{V, W \in I} (\mathbf{x}_V\mathbf{y}_W)^{\oplus N_{V, W}^U}\right)U\right)  \\
&=&
\sum_{\vv{\lambda} \in \mathcal{P}^I} \left( \prod_{U \in I} s_{\vv{\lambda}(U)}\left(\mathbf{x}_U, \mathbf{y}_U, \bigoplus_{V, W \in I} (\mathbf{x}_V\mathbf{y}_W)^{\oplus N_{V,W}^U}\right) \right) \otimes Z_{\vv{\lambda}}
\end{eqnarray*}
\end{proof}
\noindent
By comparing coefficients of Schur functions on each side of the previous theorem, we obtain the following corollary.
\begin{corollary} \label{multiplication_duality}
We have 
\[
Z_{\vv{\mu}}Z_{\vv{\nu}} = \sum_{\vv{\lambda}} a_{\vv{\mu},\vv{\nu}}^{\vv{\lambda}}Z_{\vv{\lambda}},
\]
where $a_{\vv{\mu},\vv{\nu}}^{\vv{\lambda}}$ is the coefficient of
\[
\prod_{V \in I} s_{\vv{\mu}(V)}(\mathbf{x}_V) \times \prod_{W \in I} s_{\vv{\nu}(W)}(\mathbf{y}_W)
\]
in
\[
\prod_{U \in I} s_{\vv{\lambda}(U)}\left(\mathbf{x}_U, \mathbf{y}_U, \bigoplus_{V, W \in I} (\mathbf{x}_V\mathbf{y}_W)^{\oplus N_{V,W}^U}\right).
\]
\end{corollary}
\begin{theorem}
The multiplication induced from $\mathcal{G}_\infty^{\mathbb{Q}}(R)$ makes $\mathcal{G}_\infty^{\mathbb{Z}}(R)$ into a ring.
\end{theorem}

\begin{proof}
If $\vv{\lambda}$ is the empty multipartition, then $Z_{\vv{\lambda}}$ is the identity. Thus, it suffices to show that $\mathcal{G}_\infty^{\mathbb{Z}}(R)$ is closed under multiplication, which follows from Corollary \ref{multiplication_duality}
\end{proof}

\begin{remark}
We have shown that the ring $\mathcal{G}_\infty^{\mathbb{Z}}(R)$ is an integral form of $\mathcal{G}_\infty^{\mathbb{Q}}(R)$.
\end{remark}


\begin{definition}
Let $e_r(U) = Z_{\vv{\lambda}}$, where $\vv{\lambda}(U) = (1^r)$ for a single $U \in I$ and $\vv{\lambda}(U)$ is the empty partition for all other $U$. For fixed $U$, we write $E_U(t) = \sum_{r \geq 0} e_r(U)t^r$.
\end{definition}
\noindent
The following proposition shows that $e_r(U)$ does not depend on the choice of basis $I$.
\begin{proposition}
Fix $U \in I$. We have the following equality:
\[
E_U(t) = \prod_{i \geq 1}\Theta_i (1 - (-t)^i U).
\]
\end{proposition}
\begin{proof}
Consider the generating function defining the $Z_{\vv{\lambda}}$:
\[
\sum_{\vv{\lambda} \in \mathcal{P}^I} \left( \prod_{U \in I} s_{\vv{\lambda}(U)}^{(U)} \right) \otimes Z_{\vv{\lambda}} =
\prod_{l=1}^{\infty}\Theta_l \left(1 + \sum_{U \in I} p_l^{(U)}U\right).
\]
For all $V \in I$ different from $U$, we set the variable set $\mathbf{x}_V$ to zero. This has the effect of sending any term with $|\vv{\lambda}(V)| > 0$ to zero. If we write $Z_{\lambda, U}$ for $Z_{\vv{\mu}}$ where $\vv{\mu}(U) = \lambda$ and $\vv{\mu}(V)$ is the empty partition for $V \neq U$, this gives
\[
\sum_{\lambda \in \mathcal{P}} s_{\lambda}^{(U)} \otimes Z_{\lambda, U} =
\prod_{l=1}^{\infty}\Theta_l \left(1 + p_l^{(U)}U\right).
\]
We apply the involution $\omega$ in the symmetric function variables $\mathbf{x}_U$. Recall that $\omega(s_\lambda) = s_{\lambda^\prime}$ and $\omega(p_l) = (-1)^{l-1}p_l$. Hence,
\[
\sum_{\lambda \in \mathcal{P}} s_{\lambda^\prime}^{(U)} \otimes Z_{\lambda, U} =
\prod_{l=1}^{\infty}\Theta_l \left(1 + (-1)^{l-1}p_l^{(U)}U\right).
\]
We evaluate the set of variables $\mathbf{x}_U = (t,0,0,\ldots)$. Note that $s_{\lambda^\prime}(t,0,0,\ldots)$ is equal to $t^n$ if $\lambda^\prime = (n)$, and zero if $\lambda^\prime$ has more than one part, while $p_l(t,0,0,\ldots) = t^l$. Now,
\[
\sum_{r \geq 0} e_r(U)t^r =
\prod_{l=1}^{\infty}\Theta_l \left(1 - (-t)^l U\right).
\]
\end{proof}


\begin{definition}\label{E_function}
For arbitrary $U \in R$ (not necessarily an element of $I$), we define $e_r(U)$ and $E_U(t)$ via the preceding series:
\[
E_U(t) = \sum_{r \geq 0} e_r(U) t^r = \prod_{l \geq 1}\Theta_l (1 - (-t)^l U).
\]
\end{definition}
\begin{proposition}
For $r \in \mathbb{Z}_{>0}$, and $U \in I$, the $e_r(U)$ generate $\mathcal{G}_\infty^{\mathbb{Z}}(R)$.
\end{proposition}
\begin{proof}
Passing to the associated graded algebra, the image of $e_r(U)$ is $e_r^{(U)} \in \bigotimes_{U \in I} \Lambda^{(U)}$. Because the $e_r^{(U)}$ generate the associated graded algebra, it follows the $e_r(U)$ generate $\mathcal{G}_\infty^{\mathbb{Z}}(R)$.
\end{proof}
\noindent
We will concern ourselves with understanding the relations between the $e_r(U)$.

\section{The Commutation Relation for $E_U(t)$}

\begin{lemma} \label{commuting_lemma}
If $U$ and $V$ commute in $R$, then $e_i(U)$ and $e_j(V)$ commute in $\mathcal{G}_\infty^\mathbb{Z}(R)$ for any $i$ and $j$.
\end{lemma}
\begin{proof}
It suffices to check that the generating functions $E_U(t)$ and $E_V(s)$ commute:
\begin{eqnarray*}
E_U(t)E_V(s) &=& \prod_l \Theta_l(1 - (-t)^l U) \Theta_l(1 - (-s)^l V) \\
&=& \prod_l \Theta_l \left( (1 - (-t)^l U)(1 - (-s)^l V) \right) \\
&=& \prod_l \Theta_l \left( (1 - (-s)^l V)(1 - (-t)^l U) \right) \\
&=& \prod_l \Theta_l(1 - (-s)^l V) \Theta_l(1 - (-t)^l U) \\
&=& E_V(s)E_U(t).
\end{eqnarray*}
\end{proof}

\noindent
The following lemma follows immediately from the multiplicativity of $\Theta_l$.
\begin{lemma}\label{inverse}
The multiplicative inverse of $\Theta_l(1 + t U)$ is equal to $\Theta_l((1 + t U)^{-1})$.
\end{lemma}

\begin{proposition} \label{relation}
Let $X$ and $Y$ be non-commuting variables. Then, the sum of all monomials that alternate in $X$ and $Y$ can be written in the following two ways:
\[
(1+X)(1-YX)^{-1}(1+Y) = (1+Y)(1-XY)^{-1}(1+X)
\]
\end{proposition}
\begin{proof}
Let $A = (1+X)$ and $B = (1+Y)$. Then we must show that 
\[
A(A+B-BA)^{-1}B = B(A+B-AB)^{-1}A.
\]
But the inverse of this equation is 
\[
B^{-1}(A+B-BA)A^{-1} = A^{-1}(A+B-AB)B^{-1},
\]
 and both sides equal $A^{-1}+B^{-1}+1$.
\end{proof}

\begin{theorem} \label{relation_1}
We have the following relation between generating functions (in $\mathcal{G}_\infty^\mathbb{Z}(R)[[u, v]]$):
\[
E_U(u)E_{VU}(-uv)^{-1}E_V(v) = E_V(v)E_{UV}(-uv)^{-1}E_U(u)
\]
\end{theorem}
\begin{proof}
Recall that $E_U(u) = \prod_{l \geq 1} \Theta_l(1 - (-u)^i U )$. We will work with terms corresponding to a fixed $l$, and then multiply these together to recover the result for the full generating functions. With Proposition \ref{relation} we obtain:
\begin{eqnarray*}
& &\Theta_l(1 -(-u)^l U) \Theta_l((1 - (-v)^l(-u)^l VU)^{-1}) \Theta_l(1 -(-v)^l V) \\
&=& \Theta_l((1 -(-u)^l U)(1 - (-v)^l(-u)^l VU)^{-1}(1 -(-v)^l V)))) \\
&=& \Theta_l((1 -(-v)^l V)(1 - (-u)^l(-v)^lUV)^{-1}(1 -(-u)^l U)) \\
&=&\Theta_l(1 -(-v)^l V) \Theta_l((1 - (-u)^l(-v)^l UV)^{-1}) \Theta_l(1 -(-u)^l U))) \\
\end{eqnarray*}
We may now take the product of these across $i \geq 1$. Because $\Theta_{l_1}(-)$ and $\Theta_{l_2}(-)$ commute whenever $l_1 \neq l_2$, it does not matter in what order we take the product. By taking the product over $l$, we may use Definition \ref{E_function} to write:
\[
E_U(u)E_{VU}(-vu)^{-1}E_V(v) = E_V(v)E_{UV}(-uv)^{-1}E_U(u)
\]
\end{proof}
\begin{definition}
For $W \in R$ and $n \geq 0$, let $h_n(W) \in \mathcal{G}_\infty^\mathbb{Z}(R)$ be given by determinant of the matrix $(e_{1 + j -i}(W))_{i,j=1}^{n}$ (where $e_r(W) = 0$ if $r < 0$).
\end{definition}
\noindent
Note that $h_n(W)$ is a polynomial in $e_1(W), e_2(W), \ldots , e_n(W)$ that lies in filtration degree $n$ of $\mathcal{G}_\infty^\mathbb{Z}(R)$.
\begin{lemma}
The series $\sum_{n \geq 0}h_n(W)t^n$ is the multiplicative inverse of $E_W(-t)$.
\end{lemma}
\begin{proof}
Recall that in the theory of symmetric functions, we have the relation $H(t)E(-t) = 1$, where $H(t)$ is the generating function of complete symmetric functions, and $E(t)$ is the generating function of elementary symmetric functions. Since the elementary symmetric functions are algebraically independent, we may specialise the $r$-th elementary symmetric function to $e_r(W)$ (which commute with each other by Lemma \ref{commuting_lemma}), which specialises $E(-t)$ to $E_W(-t)$. The coefficients of $H(t)$ are given by the Jacobi-Trudi formula: $h_n = \det(e_{1+j-i})_{i,j=1}^{n}$, which specialise to $h_n(W)$ as defined above.
\end{proof}
\begin{corollary} \label{first_degree_relation}
The commutation relation between $e_i(U)$ and $e_j(V)$ is given by the following:
\[
\sum_{k = 0}^{\min(i,j)}e_{i-k}(U) h_{k}(VU) e_{j-k}(V) = \sum_{k = 0}^{\min(i,j)} e_{j-k}(V) h_k(UV) e_{i-k}(U)
\]
More explicitly:
\[
[e_i(U), e_j(V)] = \sum_{k=1}^{\min(i,j)} e_{j-k}(V)h_k(UV)e_{i-k}(U) - e_{i-k}(U)h_k(VU)e_{j-k}(V).
\]
\end{corollary}
\noindent
Note in particular that because $h_k(W)$ is in filtration degree $k$, the right hand side is contained in filtration degree $i+j-1$.
\begin{proof}
This equation is simply the coefficient of $u^iv^j$ in the equation in Theorem \ref{relation_1}.
\end{proof}

\begin{example} \label{example_1}
The case of $i=j=1$ in Corollary \ref{first_degree_relation} is $e_1(U)e_1(V) + e_1(VU) = e_1(V)e_1(U) + e_1(UV)$.
\end{example}
\begin{corollary} \label{monomial_algebra}
Suppose that $R$ has a $\mathbb{Z}$-basis $I$ such that the product of any two basis elements is either zero or another basis element (a ``monomial algebra''). Then $\mathcal{G}_\infty^{\mathbb{Z}}(R)$ is generated by $e_i(U)$ for $U \in I$ and $i \geq 1$. If we let $U,V \in I$, the relations are simply:
\[
E_U(u)E_{VU}(-vu)^{-1}E_V(v) = E_V(v)E_{UV}(-uv)^{-1}E_U(u)
\]
where $E_{0}(t) = 1$.
\end{corollary}
\begin{proof}
Recall that the associated graded algebra of $\mathcal{G}_\infty^{\mathbb{Z}}(R)$ is generated as a free polynomial algebra by the $e_i(U)$. Hence, the $e_i(U)$ generate $\mathcal{G}_\infty^{\mathbb{Z}}(R)$, and to describe the relations, it suffices to give expressions for the commutators $[e_i(U), e_j(V)]$ that lie in filtration degree $i+j-1$. This is done by Corollary \ref{first_degree_relation}, where the essential point is that $UV$ and $VU$ are either zero or in $I$, so that $h_k(UV)$ and $h_k(VU)$ are polynomials in elements of our generating set.
\end{proof}
\begin{example}
Suppose that $R = \mathbb{Z}G$, the integral group algebra of $G$. Taking $I=G$, we are in the case of Corollary \ref{monomial_algebra}. We conclude that $\mathcal{G}_\infty^{\mathbb{Z}}(kG)$ is generated by $e_i(g)$ for $i \in \mathbb{Z}_{>0}$ and $g \in G$, subject to the relations arising from the following equations of generating functions for $g, h \in G$ (where, as always, $E_k(t) = \sum_{i \geq 0} e_i(k)t^i$):
\[
E_g(u)E_{hg}(-vu)^{-1}E_h(v) = E_h(v)E_{gh}(-uv)^{-1}E_g(u).
\]
\end{example}
\begin{example}
Suppose that $R = \mbox{Mat}_{n}(\mathbb{Z})$, and $I$ is the set of elementary matrices $\{ \mathbf{E}_{i,j} \}_{i,j=1}^n$. Because $\mathbf{E}_{i,k} \mathbf{E}_{k,l} = \delta_{j,k}\mathbf{E}_{i,l}$, this is a monomial algebra. It follows $\mathcal{G}_\infty^{\mathbb{Z}}(\mbox{Mat}_n(\mathbb{Z}))$ has generators $e_n(\mathbf{E}_{ij})$ and relations
\[
{E_{i,j}}(u) {E_{k,j}}(-vu)^{\delta_{li}}{E_{k,l}}(v) = {E_{k,l}}(v) {E_{i,l}}(-uv)^{\delta_{jk}} {E_{i,j}}(u)
\]
where ${E_{i,j}}(t) = \sum_{n \geq 0}e_n(\mathbf{E}_{ij})t^n$.
\end{example}

\section{Decomposing $E_U(t)$ when $U \notin I$}
\noindent
In the previous section, we computed the commutation relations of the generators $e_i(U)$ at the cost of introducing terms of the form $e_i(W)$, where $W$ might not be an element of $I$. In this section, we express $e_i(W)$ in terms of $e_i(U)$ for $U \in I$. 
\begin{definition}
Let $\mu(r)$ denote the usual M\"{o}bius function (so $\sum_{d | n} \mu(d) = \delta_{1,n}$, the Kronecker delta). For $U \in R$, let $F_U(t)$ be defined by:
\[
F_U(t) = - \sum_{r \geq 1} \frac{\mu(r)}{r} \log(E_{U^r}(-t^r))
\]
\end{definition}
\noindent
Note that this series does not have integral coefficients; it lies in $\mathcal{G}_\infty^{\mathbb{Q}}(R)[[t]]$. We now express $F_U(t)$ in terms of $T_i(U^r)$.
\begin{proposition} \label{F_proposition}
We have the following equality of elements of $\mathcal{G}_\infty^{\mathbb{Q}}(R)[[t]]$:
\[
F_U(t) = \sum_{i \geq 1} T_i(U) t^i
\]
\end{proposition}
\begin{proof}
We substitute Definition \ref{E_function} into the definition of $F_U(t)$:
\begin{eqnarray*}
F_U(t) &=& - \sum_{r \geq 1} \frac{\mu(r)}{r} \log(E_{U^r}(-t^r))\\
&=& - \sum_{r \geq 1}\frac{\mu(r)}{r}\sum_{i \geq 1} T_i(\log(1 - t^{ir} U^r)) \\
&=& - \sum_{r \geq 1}\frac{\mu(r)}{r}\sum_{i \geq 1} T_i(- \sum_{d \geq 1} \frac{t^{ird}U^{rd}}{d}) \\
&=& \sum_{r \geq 1} \sum_{i \geq 1} \sum_{d \geq 1} \frac{\mu(r)t^{ird}}{rd}T_i(U^{rd}) \\
&=& \sum_{i \geq 1} \sum_{n \geq 1} \sum_{r | n} \frac{\mu(r)t^{in}}{n}T_i(U^n) \hspace{20mm} \mbox{(change variables to $n = rd$)}\\
&=&\sum_{i \geq 1} \sum_{n \geq 1} \delta_{1, n} \frac{t^{in}}{n}T_i(U^n) \\
&=&\sum_{i \geq 1} t^{i}T_i(U).
\end{eqnarray*}
In the second-last step, we used that $\sum_{r | n} \mu(r) = \delta_{1,n}$ (Kronecker delta).
\end{proof}

\begin{theorem} \label{relation_2}
Suppose that $W = \sum_i a_i U_i$, with $a_i \in \mathbb{Z}$ and $U_i \in R$. Then we have:
\[
F_W(t) = \sum_i a_i F_{U_i}(t).
\]
\end{theorem}
\begin{proof}
This immediately follows from the fact that $T_r(\sum_i a_i U_i) = \sum_i a_i T_r(U_i)$ and Proposition \ref{F_proposition}, expressing $F_U(t)$ in terms of $T_r(U)$.
\end{proof}
\noindent
If we write $W = \sum_{U \in I} a_U U$, we may extract the coefficient of $t^{n}$ in the equation in Theorem \ref{relation_2} to express $e_n(W)$ in terms of $e_n(U)$ (our generators of interest) plus terms with smaller values of $n$. By induction, this allows us to express $e_n(W)$ for arbitrary $W$ in terms of $e_r(U)$ for $U \in I$. 
\begin{example} \label{relation_eg}
We calculate $F_U(t)$ to order $t^2$:
\begin{eqnarray*}
F_U(t)
&=& e_1(U)t + \frac{1}{2}(e_1(U)^2 - e_1(U^2) - 2e_2(U))t^2 + \cdots
\end{eqnarray*}
Let us take $W = U_1 + U_2$. Then the degree 1 and 2 terms of the equation in Theorem \ref{relation_2} become:
\begin{eqnarray*}
e_1(U_1 + U_2) &=& e_1(U_1) + e_1(U_2) \\
\frac{1}{2}(e_1(U_1 + U_2)^2 -e_1((U_1 + U_2)^2) - 2e_2(U_1 + U_2)) &=&\frac{1}{2}( e_1(U_1)^2 -e_1(U_1^2) - 2e_2(U_1))  \\
& &+\frac{1}{2}(e_1(U_2)^2 -e_1(U_2^2) - 2e_2(U_2)).
\end{eqnarray*}
These rearrange to the following:
\begin{eqnarray*}
e_1(U_1 + U_2) &=& e_1(U_1) + e_1(U_2) \\
e_2(U_1 + U_2) &=& \frac{1}{2}(e_1(U_1)e_1(U_2)+e_1(U_2)e_1(U_1)  -e_1(U_1U_2) - e_1(U_2U_1))   +e_2(U_1)  + e_2(U_2) \\
\end{eqnarray*}
Note that the previous equation for $e_2(U_1 + U_2)$ is not manifestly integral. However, the relation $e_1(U_1)e_1(U_2) - e_1(U_1U_2) = e_1(U_2)e_1(U_1) - e_1(U_2U_1)$ from Example \ref{example_1} gives the following:
\[
e_2(U_1 + U_2) = e_1(U_1)e_1(U_2) - e_1(U_1U_2) + e_2(U_1) + e_2(U_2) = e_1(U_2)e_1(U_1) - e_1(U_2U_1) + e_2(U_1) + e_2(U_2)
\]
Thus $e_2(U_1 + U_2)$ can be written in terms of $e_i(V)$ for $V \in I$ (after possibly decomposing $e_1(U_1U_2)$ or $e_1(U_2U_1)$), although not in a canonical way.
\end{example}
\noindent
We have the following theorem.
\begin{theorem} \label{new_presentation}
The $\mathbb{Q}$-algebra $\mathcal{G}_\infty^{\mathbb{Q}}(R)$ admits the following presentation. The generators $e_r(U)$ are $U \in I$ and $r \in \mathbb{Z}_{>0}$. We write $E_U(t) = \sum_{i \geq 0} e_i(U)t^i$, where $e_0(U)$ is taken to be the multiplicative identity. The relations between the generators are given by the following equalities of generating functions.
\begin{equation} \label{mult_reln}
E_U(u)E_{VU}(-uv)^{-1}E_V(v) = E_V(v)E_{UV}(-uv)^{-1}E_U(u)
\end{equation}
If $W = \sum_{U \in I} a_U U$ in $R$ (where $a_U \in \mathbb{Z}$), we also have:
\begin{equation} \label{sum_reln}
\sum_{r \geq 1} \frac{\mu(r)}{r} \log(E_{W^r}(-t^r))  = 
\sum_{U \in I} a_U \sum_{r \geq 1} \frac{\mu(r)}{r} \log(E_{U^r}(-t^r))
\end{equation}
Further $\mathcal{G}_\infty^{\mathbb{Z}}(R)$ is the $\mathbb{Z}$-subalgebra of $\mathcal{G}_\infty^{\mathbb{Q}}(R)$ generated by the $e_i(U)$.
\end{theorem}
\begin{proof}
As in Corollary \ref{monomial_algebra}, we use the fact that the $e_i(U)$ (for $U \in I$) generate the associated graded algebra, and the fact that Theorem \ref{relation_1} expresses the commutator $[e_i(U), e_j(V)]$ as an element in filtration degree $i+j-1$. However, unlike Corollary \ref{monomial_algebra}, this now involves the generators $e_k(UV)$, where $UV$ may not be an element of $I$. To provide a presentation, it suffices to express $e_k(UV)$ in terms of $e_r(W)$ for $W \in I$, which is precisely what is achieved by Theorem \ref{relation_2}.
\end{proof}
\noindent
\begin{remark}
This falls short of giving a presentation of $\mathcal{G}_\infty^{\mathbb{Z}}(R)$ because the second family of relations are not manifestly integral. One may extract the coefficient of $t^r$ and rearrange for $e_r(U)$, as was done for $r=2$ in Example \ref{relation_eg}. Because the resulting expression is in $\mathcal{G}_\infty^{\mathbb{Z}}(R)$, it must be possible to rewrite $e_r(U)$ as a linear combination of monomials in $e_s(V)$ for $s \leq r$ and $V \in I$. Knowing these relations would be sufficient to give a presentation. It is reasonable to expect that it is possible to give an integral expression in terms of monomials in $e_{s}(U)$ (where the $U$ are not necessarily in $I$) as in Example \ref{relation_eg} for $r=2$. However, as was already noted, there is no canonical choice of such a decomposition.
\end{remark}

\section{$\lambda$-Ring Structure}
\noindent
Before introducing a $\lambda$-ring structure on $\mathcal{G}_\infty^{\mathbb{Z}}(R)$, we prove a statement about exterior powers for wreath products.
\begin{lemma} \label{wedgelemma}
Let $U$ be an object of a symmetric tensor category $\mathcal{C}$ (over an algebraically closed field of characteristic zero). Consider the following object of the wreath product category $\mathcal{W}_n(\mathcal{C}) = (\mathcal{C}^{\boxtimes n})^{S_n}$:
\[
V = \Ind{S_1 \times S_{n-1}}{S_n}(U \boxtimes \mathbf{1}^{\boxtimes (n-1)}).
\]
Then, the $r$-th exterior power of $V$ is given by the following formula:
\[
\bigwedge\nolimits^r (V) = \bigoplus_{\substack{\lambda \vdash r \\ l(\lambda) \leq n}} \Ind{S_{n - l(\lambda)} \times \prod_i S_{m_i}}{S_n}(\mathbf{1}^{n -l(\lambda)} \boxtimes (\boxtimes_i (\bigwedge\nolimits^i(U)^{\otimes m_i} \otimes \sigma_{m_i}^{(i)}))).
\]
Here $\sigma_{m_i}^{(i)}$ is the trivial representation of the symmetric group $S_{m_i}$ if $i$ is even, and the sign representation of $i$ is odd.
\end{lemma}
\begin{proof}
We sketch the proof. We perform the calculation by describing an exterior power of $V$ as the image of an antisymmetrising morphism of a tensor power of $V$. More precisely, $V^{\otimes r}$ carries an action of the symmetric group $S_r$ by permuting the tensor factors, and $\bigwedge\nolimits^r (V)$ is the image of the endomorphism defined by $\sum_{g \in S_r} \varepsilon(g) g$, where $\varepsilon(g)$ is the sign of $g$.
\newline \newline \noindent
Consider $V^{\otimes r}$, which we simplify using the Mackey formula. Recall that a set partition of a set $X$ is a set of subsets of $X$ whose disjoint union is equal to $X$. Then, we claim:
\begin{equation} \label{wedgetensor}
V^{\otimes r} = \bigoplus_{\substack{\mbox{set partitions $D$ of } \{1,2, \cdots, r\} \\ \mbox{having at most $n$ parts}}} \Ind{S_{n - |D|} \times \prod_{d \in D} S_1}{S_n}(\mathbf{1}^{\boxtimes (n - |D|)} \boxtimes (\boxtimes_{d \in D} U^{\otimes |d|}))
\end{equation}
We demonstrate this by induction on $r$. When $r = 1$, there is only one set partition of $\{1\}$, namely $\{\{1\}\}$, and this gives rise to a single summand which is isomorphic to $V$. For the inductive step, we must tensor with $V$. To use the Mackey formula
\[
\Ind{H}{G}(M) \otimes \Ind{K}{G}(N) = \bigoplus_{s \in H \backslash G / K} \Ind{H \cap sKs^{-1}}{G} (M \otimes sN),
\]
we must know the double cosets $(S_{n - |D|} \times \prod_{d \in D} S_1) \backslash S_n / (S_{1} \times S_{n-1})$. If we view $S_{1} \times S_{n-1}$ as the subgroup of $S_n$ fixing $1 \in \{1,2,\cdots, n \}$, then the double cosets are in bijection with the nontrivial factors in the product $(S_{n - |D|} \times \prod_{d \in D} S_1)$ (so there are $|D| + 1$ if $n > |D|$ and $|D|$ otherwise) where the double coset associated to a factor group is the set of all elements in $S_n$ mapping $1 \in \{1,2,\cdots, n\}$ to an element permuted by that factor group (when considered as a subgroup of $S_n$). Thus by the Mackey formula, tensoring an object associated to a set partition $D$ with $V$ gives a sum of induced objects associated to all set partitions obtained from $D$ by adding $r+1$ to any part $d \in D$, or in a part of its own if $|D| < n$. When we consider all set partitions $D$ of $\{1,2,\cdots, r\}$ of length at most $n$, this process of appending $r+1$ gives all set partitions of $\{1, 2, \cdots, r+1\}$ of length at most $n$.
\newline \newline \noindent
The action of $S_r$ on our decomposition in Equation \ref{wedgetensor} is given by the obvious action of $S_r$ on set partitions of $\{1, 2, \cdots, r\}$, provided that in the product $\boxtimes_{d \in D} U^{\otimes |d|}$, we identify each tensor factor of $U^{\otimes |d|}$ with an element of $d$, and that these tensor factors are also permuted according to the symmetric group action. Two summands lie in the same orbit of $S_r$ if and only if they have the same multisets $\{|d| \mid d \in D\}$ for their respective set partitions $D$. Such a multiset defines a partition $\lambda$ by reordering the elements of the multiset in nonincreasing order. It remains to show that the image of the antisymmetrising morphism on the summands in the orbit corresponding to $\lambda$ gives the term indexed by $\lambda$ in the statement of the Lemma.
\newline \newline \noindent
Note that because the sign of a permutation $g$, $\varepsilon(g)$, is multiplicative, we have the following factorisation property. Let $H$ be a subgroup of $G$ (itself a subgroup of $S_r$), and $G/H$ a collection of left coset representatives. Then:
\[
\sum_{g \in G} \varepsilon(g)g = \left(\sum_{c \in G/H}\varepsilon(c)c \right) \left( \sum_{h \in H} \varepsilon(h) h \right).
\]
We consider a set partition $D$ with associated partition $\lambda$, and take $H = \prod_i S_i \wr S_{m_i}$ (the wreath product $S_i \wr S_{m_i}$ is considered as a subgroup of $S_{im_i}$ in the usual way). To find the image of $U^{\boxtimes |D|}$ under this partial antisymmetrisation consider the case of a single factor $S_i \wr S_{m_i}$ acting on $(U^{\otimes i})^{\boxtimes m_i}$. We further factorise the antisymmetrising morphism taking $H$ to be the subgroup of $S_i^{m_i}$ in the wreath product $G = S_i \wr S_{m_i}$. This acts by antisymmetrising each of the $\boxtimes$-tensor factors; the image is $\bigwedge^i(U)^{\boxtimes m_i}$. The complementary antisymmetriser $\sum_{c \in G/H}\varepsilon(c)c $ coming from coset representatives $(S_i \wr S_{m_i}) / S_i^{m_i}$ is indexed by elements of $S_{m_i}$, and the sign of an element $g$ is found as follows. Note that a transposition in $S_{m_i}$ permutes each the factor groups of $S_{i}^{m_i}$. As an element of $S_{im_i}$, this has cycle type $1^{i(m_i - 2)}2^{i}$, which is odd if $i$ is odd and is even if $i$ is even. Because transpositions generate $S_{m_i}$ and signs are multiplicative, it follows that the complementary antisymmetriser, viewed as an element of the group algebra of $S_{m_i}$ is (up to a scalar) the antisymmetriser if $i$ is odd, and the symmetriser if $i$ is even. Thus, after antisymmetrising over $S_i \wr S_{m_i}$, we obtain $\bigwedge\nolimits^i(U)^{\otimes m_i} \otimes \sigma_{m_i}^{(i)}$. Performing this calculation for all $i$ simultaneously gives a term isomorphic to the summand indexed by $\lambda$ in the statement of the Lemma. It remains to explain how the remaining antisymmetrisation gives rise to exactly one summand for each $\lambda$.
\newline \newline \noindent
For each $\lambda$, the number of summands we have is $N = |S_r / \prod_{i} S_i \wr S_{m_i}|$. The antisymmetriser
\[
\sum_{c \in S_n/\prod_i S_i \wr S_{m_i}} \varepsilon(c)c 
\]
defines a map from one summand to the direct sum of all the $N$ summands associated to $\lambda$. Taking into consideration that the choice of the subgroup $\prod_i S_i \wr S_{m_i}$ depended on the choice of the original set partition $D$, one can check that all $N$ maps obtained in this way have the same image (which can be identified with the summand corresponding to $\lambda$ in the statement of the Lemma).
\end{proof}
\noindent
Recall that a $\lambda$-ring structure on a commutative ring $R$ is a collection of operations $\lambda^i:R \to R$ ($i \in \mathbb{Z}_{\geq 0}$) satisfying $\lambda_0(x) = 1$, $\lambda^1(x) = x$ and the following three compatibility conditions \cite{knutson}:
\[
\lambda^n(x+y) = \sum_{i + j = n} \lambda^i(x)\lambda^j(y)
\]
\[
\lambda^n(xy) = P_n(\lambda^1(x), \lambda^2(x) \cdots, \lambda^n(x), \lambda^1(y), \lambda^2(y) \cdots, \lambda^n(y))
\]
(Here $P_n$ are specific polynomials with integer coefficients.)
\[
\lambda^m (\lambda^n(x)) = Q_{m,n}(\lambda^1(x), \lambda^2(x) \cdots, \lambda^{mn}(x))
\]
(Here $Q_{m,n}$ are specific polynomials with integer coefficients.) Let $a$ be an element of a $\lambda$-ring, and $\lambda_t(a) = \sum_{n \geq 0} \lambda^n(a)t^n$, the generating function of the $\lambda$ operations applied to $a$. Also let $\psi_t(a) = \sum_{n \geq 1} \psi_n(a)t^{n-1}$ be the generating function defined by the following equation:
\[
\frac{d}{dt}\log(\lambda_t(a)) = \psi_{-t}(a).
\]
The $\psi_n$ are called Adams operations; they are are commuting endomorphisms of $R$ satisfying $\psi_n \circ \psi_m = \psi_{nm}$ \cite{knutson}.
\newline \newline \noindent
We are now able to describe a $\lambda$-ring structure on $\mathcal{G}_\infty^{\mathbb{Z}}(R)$.
\begin{theorem}
\begin{enumerate}
\item Suppose that there is a $\lambda$-ring structure on $R$. Then, there is a $\lambda$-ring structure on $\mathcal{G}_\infty^{\mathbb{Z}}(R)$ defined on $e_1(U)$ by:
\[
\sum_{n=0} \lambda^{n}(e_1(U)) t^n = \prod_{l \geq 1} \Theta_l \left( \sum_{r \geq 0}(-1)^{r(l-1)}t^{rl} \lambda^{r}(U)\right)
\]
Together with the $\lambda$-ring axioms, this determines the $\lambda$-ring structure on all of $\mathcal{G}_\infty^{\mathbb{Z}}(R)$. Over $\mathcal{G}_\infty^{\mathbb{Q}}(R)$, the $\lambda$-ring structure is given by the following Adams operations (we write $\psi_d$ for the Adams operations on $R$):
\[
\Psi_m(T_l(U)) = \sum_{\substack{d \mid m \\ \gcd(d, l) = 1}} \frac{m}{d} T_{l \frac{m}{d}}(\psi_d(U))
\]
\item Consider the case where $\mathcal{C}$ is a symmetric tensor category, so that $S_t(\mathcal{C})$ is also a symmetric tensor category, inducing a $\lambda$-ring structure on $\mathcal{G}(S_t(\mathcal{C})) = \mathcal{G}_\infty^{\mathbb{Z}}(\mathcal{G}(\mathcal{C}))$. This $\lambda$-ring structure is obtained in the above way from the $\lambda$-ring structure induced on $\mathcal{G}(\mathcal{C})$ by the symmetric tensor structure.
\end{enumerate}
\end{theorem}
\begin{proof}
We initially work over $\mathcal{G}_\infty^{\mathbb{Q}}(R)$, and deduce integrality at the end. Following \cite{knutson} a $\lambda$-ring structure on a torsion-free ring is uniquely defined by the Adams operations $\psi_m$. We write the Adams operations on $R$ as $\psi_m$, and we will use them to construct Adams operations on $\mathcal{G}_\infty^{\mathbb{Q}}(R)$, which we denote $\Psi_m$.
\newline \newline \noindent
When $R$ is a $\lambda$-ring, it is in particular commutative. This means that $\mathcal{G}_\infty^{\mathbb{Q}}(R)$ is the free polynomial algebra generated by $T_l(U)$ for $l \in \mathbb{Z}_{>0}$ and $U \in I$ (the universal enveloping algebra of $\mathbb{Q} \otimes_{\mathbb{Z}} \mathcal{G}(R)$ is itself a free polynomial algebra, and $\mathcal{G}_\infty^{\mathbb{Q}}(R)$ is a tensor product of these). Thus, to define an endomorphism of $\mathcal{G}_\infty^{\mathbb{Q}}(R)$, it is enough to define the image of each $T_l(U)$. We consider $\Psi_m$ as in the statement of the theorem.
Note that $\Psi_1$ is the identity map. We also have:
\begin{eqnarray*}
\Psi_m(\Psi_n(T_l(U))) &=& \sum_{\substack{d \mid m \\ \gcd(d, l) = 1}} \frac{m}{d} \Psi_{n}(T_{l \frac{m}{d}}(\psi_d(U))) \\
&=& \sum_{\substack{d \mid m \\ \gcd(d, l) = 1}} \frac{m}{d}  \sum_{\substack{e \mid n \\ \gcd(e, lm/d) = 1}} \frac{n}{e} T_{l \frac{m}{d}\frac{n}{e}}(\psi_e(\psi_d(U))) \\
&=& \sum_{\substack{d \mid m \\ \gcd(d, l) = 1}}  \sum_{\substack{e \mid n \\ \gcd(e, lm/d) = 1}} \frac{mn}{de} T_{l \frac{mn}{de}}(\psi_{de}(U)) \\
&=&  \sum_{\substack{f \mid mn \\ \gcd(f, l) = 1}}   \frac{mn}{f} T_{l \frac{mn}{f}}(\psi_{f}(U)) \hspace{20mm} \mbox{(change variables to $f = de$)}\\
&=& \Psi_{mn}(T_l(U))
\end{eqnarray*}
This means that we have a valid collection of Adams operations, and thus a unique $\lambda$-ring structure.
\newline \newline \noindent
We set $a = e_1(U) = T_1(U)$. Then $\Psi_n(a) = \sum_{d \mid n} \frac{n}{d} T_{\frac{n}{d}}(\psi_d(U))$, and we may use this to calculate the generating function of the Adams operations.
\begin{eqnarray*}
\Psi_t(T_1(U)) &=& \sum_{n \geq 1} \sum_{d \mid n} \frac{n}{d} T_{\frac{n}{d}}(\psi_d(U)) t^{n-1} \\
&=& \sum_{d \geq 1} \sum_{r \geq 1} r T_{r}(\psi_d(U)) t^{rd-1} \hspace{20mm} \mbox{(change variables to $rd = n$)}\\
&=& t^{r-1} \sum_{r \geq 1} r T_r(\psi_{t^r}(U)) \\
&=& t^{r-1} \sum_{r \geq 1} r T_r(\frac{d}{d(-t^r)} \log(\lambda_{-t^r}(U))) \\
&=& -\frac{d}{dt}  \sum_{r \geq 1} T_r(\log(\lambda_{-t^r}(U))) \hspace{20mm} \mbox{(chain rule)}\\
\end{eqnarray*}
We may now integrate and exponentiate to obtain a generating function for $\lambda^r(T_1(U))$.
\begin{eqnarray*}
\sum_{n \geq 0} \lambda^n(T_1(U)) t^n &=& \exp( \int \Psi_{-t}(T_1(U)) dt) \\
&=& \exp( \int \frac{d}{dt}  \sum_{r \geq 1} T_r(\log(\lambda_{-(-t)^r}(U))) dt) \\
&=& \exp(  \sum_{r \geq 1} T_r(\log(\lambda_{-(-t)^r}(U)))) \\
&=& \prod_{l \geq 1}\Theta_l\left( \sum_{r \geq 0} \lambda^r(U) (-(-t)^l)^r\right)
\end{eqnarray*}
This is the claimed formula for the $\lambda$-ring structure. To see that $\lambda^r(e_1(U))$ is integral (i.e. an element of $\mathcal{G}_\infty^{\mathbb{Z}}(\mathcal{C})$), we proceed as follows. Let $c(r,V)$ ($r \in \mathbb{Z}_{ \geq 0}$, $V \in I$) be defined by
\[
\lambda^r(U) = \sum_{V \in I} c(r,V)V.
\]
Consider an infinite collection of symmetric function variables sets, denoted $\mathbf{x}_{r,V}$ for $V \in I$ and $r \in \mathbb{Z}_{> 0}$. Take the definition of the $Z_{\vv{\lambda}}$,
\[
\sum_{\vv{\lambda} \in \mathcal{P}^I} \left( \prod_{V \in I} s_{\vv{\lambda}(V)}^{(V)} \right) \otimes Z_{\vv{\lambda}} = \prod_{l \geq 1} \Theta_l(1 + \sum_{V \in I} p_l^{(V)}V),
\]
and make the substitution $\mathbf{x}_V = \bigoplus_{r \in \mathbb{Z}_{> 0}} \mathbf{x}_{r,V}^{\oplus c(r,V)}$. The right hand side becomes
\[
\prod_{l \geq 1} \Theta_l(1 + \sum_{V \in I} p_l(\bigoplus_{r \in \mathbb{Z}_{> 0}} \mathbf{x}_{r,V}^{\oplus c(r,V)})V) = \prod_{l \geq 1} \Theta_l(1 + \sum_{V \in I}\sum_{r \in \mathbb{Z}_{> 0}} c(r,V) p_l(\mathbf{x}_{r,V})V).
\]
For each variable set $\mathbf{x}_{r,V}$, we apply the involution $\omega$ $r$ times. This has the effect of multiplying $p_l(\mathbf{x}_{r,V})$ by $(-1)^{r(l-1)}$. We now evaluate each variable set $\mathbf{x}_{r,V}$ at the values $t^r,0,0,\ldots$ (this maps $p_l(x_{r,V})$ to $t^{rl}$). Our expression becomes
\begin{eqnarray*}
\prod_{l \geq 1} \Theta_l(1 + \sum_{V \in I}\sum_{r \in \mathbb{Z}_{> 0}} c(r,V) (-1)^{r(l-1)}t^{rl}V)&=&\prod_{l \geq 1} \Theta_l(1 + \sum_{r \in \mathbb{Z}_{> 0}} (-1)^{r(l-1)}t^{rl}\sum_{V \in I}c(r,V)V) \\
&=& \prod_{l \geq 1} \Theta_l(1 + \sum_{r \in \mathbb{Z}_{> 0}} (-(-t)^l)^{r}\lambda^{r}(U)).
\end{eqnarray*}
We recognise this as the generating function defining the $\lambda$-operations on $e_1(U)$. Because we began with a generating function describing integral elements and performed operations that preserve integrality, we conclude that $\lambda^i(e_1(U))$ is integral. 
\newline \newline \noindent
Consider the following generating function in the variable sets $\mathbf{x}_r$ for $r \in \mathbb{Z}_{>0}$:
\[
\prod_{l \geq 1} \Theta_l(1 + \sum_{r \geq 1} \lambda^r(U) p_l(\mathbf{x}_r)) = \sum_{\vv{\lambda} \in \mathcal{P}^{\mathbb{Z}}} \prod_{n \in \mathbb{Z}_{>0}} s_{\vv{\lambda}(n)}^{(n)} \otimes Z_{\vv{\lambda}}.
\]
The $Z_{\vv{\lambda}}$ defined by the above are not necessarily a basis of $\mathcal{G}_\infty^{\mathbb{Z}}(R)$ because $\lambda^r(U)$ may not be a basis of $R$. They are nevertheless elements of $\mathcal{G}_\infty^{\mathbb{Z}}(R)$ (i.e. integral). By applying the involution $\omega$ to $\mathbf{x}_r$ r times, and then evaluating at $\mathbf{x}_r = (t^r,0,0, \ldots)$, our generating function becomes
\[
\prod_{l \geq 1} \Theta_l(1 + \sum_{r \geq 1} \lambda^r(U) (-(-t^l)^r)) = \sum_{\vv{\lambda}} t^{\sum_{r \geq 1} r|\vv{\lambda}(r)|} Z_{\vv{\lambda}}.
\]
Here, the sum over $\vv{\lambda}$ only includes terms where $\vv{\lambda}(r)$ consists of a single row when $r$ is even, and consists of a single column when $r$ is odd. The left hand side is the generating function of $\lambda^i(e_1(U))$. The terms in the right hand side are identified with the terms in Lemma \ref{wedgelemma} (where $n$ is taken sufficiently large). This shows that this $\lambda$-ring structure is inherited from the implied symmetric tensor category structure on $S_t(\mathcal{C})$, if the $\lambda$-ring structure on $\mathcal{G}(\mathcal{C})$ comes from a symmetric tensor category structure on $\mathcal{C}$. This is because $S_t(\mathcal{C})$ (as defined in \cite{Mori}) is defined as an interpolation of finite wreath products.
\newline \newline \noindent
To see that the $\lambda^i(e_1(U))$ uniquely determine the $\lambda$-ring structure, we prove by induction on $r$ that $e_r(U)$ is a polynomial in the variables $\lambda^i(e_1(V))$ (where $i \leq r$ and any $V$ is permitted), then the $\lambda$-ring axioms guarantee that $\lambda^j(e_r(U))$ is uniquely determined and integral. The base case, $r=1$, is immediate. By inspecting the coefficient of $t^n$ in the generating function in Definition \ref{E_function}, we see that $e_n(U)$ is equal to $T_n(U)$ plus a polynomial in $T_s(V)$ for $s<n$ and $V$ arbitrary. A similar inspection for the generating function in the statement of the theorem shows that $\lambda^r(e_1(U))$ equals $T_r(U)$ plus lower order terms. Hence, $\lambda^r(e_1(U)) - e_r(U)$ is a (necessarily integral) polynomial in $T_s(V)$ for $s < r$ and arbitrary $V$. By an upper triangularity argument, $\lambda^r(e_1(U)) - e_r(U)$ is a polynomial in $e_s(V)$ for $s < r$ and $V$ arbitrary. By induction $e_r(U)$ is a polynomial in $\lambda^{i}(e_1(V))$. We conclude that the $\lambda$-ring axioms (which define how $\lambda$-operations behave on sums and products) determine the $\lambda$-operations on $e_r(U)$ and hence all of $\mathcal{G}_\infty^{\mathbb{Q}}(\mathcal{C})$.
\end{proof}

\section{Hopf Algebra Structure}
\begin{theorem} \label{hopf_str}
There is a Hopf algebra structure $(\Delta, \epsilon, S)$ on $\mathcal{G}_\infty^{\mathbb{Z}}(\mathcal{C})$ defined on the generating functions $E_U(t)$ by:
\begin{equation}
\Delta(E_U(t)) = E_U(t) \otimes E_U(t),
\end{equation}
\begin{equation}
\epsilon(E_U(t)) = 1,
\end{equation}
\begin{equation}
S(E_U(t)) = E_U(t)^{-1}.
\end{equation}
When working over $\mathcal{G}_\infty^{\mathbb{Z}}(\mathcal{C})$, we have that the $T_i(U)$ are primitive.
\end{theorem}
\begin{proof}
The fact that $\epsilon$ and $S$ are well defined (respect the algebra structure) and satisfy the Hopf algebra axioms easily follows from the fact that $E_U(t)$ is grouplike and $\log(E_U(t))$ is primitive with respect to $\Delta$ when using the presentation in Theorem \ref{new_presentation}. Note that $F_U(t)$ is a linear combination of $\log(E_{U^r}(-t^r))$, and therefore primitive. By Proposition \ref{F_proposition}, we may take the $t^i$ coefficient of $F_U(t)$ to deduce that $T_i(U)$ is primitive.
\end{proof}
\begin{remark}
The above Hopf algebra structure is similar to that of the ring of symmetric functions, and coincides with it when $\mathcal{C} = \mbox{Vect}$. A succinct way of characterising the Hopf algebra structure is to say that the generating functions $E_U(t)$ are grouplike.
\end{remark}
\begin{remark}
The comultiplication is categorified by a ``restriction'' functor 
\[
\underline{Rep}(\mathcal{C} \wr S_{t_1 + t_2}) \to \underline{Rep}(\mathcal{C} \wr S_{t_1}) \boxtimes \underline{Rep}(\mathcal{C} \wr S_{t_2}).
\]
As the Grothendieck ring $G_\infty^{\mathbb{Z}}(\mathcal{C}) = \mathcal{G}(S_t(\mathcal{C}))$ does not depend on $t$, this gives rise to a comultiplication on $G_\infty^{\mathbb{Z}}(\mathcal{C})$. Consider the case $\mathcal{C} = \mbox{Vect}$, so that $S_t(\mathcal{C})$ is the usual Deligne category $\underline{Rep}(S_t)$. Recall that this category is the Karoubian envelope of a category $\underline{Rep}_0(S_t)$ whose objects are $V^{\otimes n}$ as discussed in \cite{OC}. The functor we are interested in, $\mathcal{F}$, satisfies $\mathcal{F}(V) = V \boxtimes 1 \oplus 1 \boxtimes V$. The wreath product case generalises this. See definition 4.21 in \cite{Mori}. One easily checks that the functor takes the antisymmetrising endomorphism of $U^{\boxtimes (m+n)}$ to the tensor product of the antisymmetrising endomorphisms of $U^{\boxtimes m}$ and $U^{\boxtimes n}$ (for all possible values of $m$ and $n$). Passing to the Grothendieck ring, we obtain a comultiplication $\Delta$, which satisfies the property that $\Delta(e_n(U)) = \sum_{i=0}^n e_i(U) \otimes e_{n-i}(U)$, namely our comultiplication.
\end{remark}
\noindent
We now express the comultplication in terms of generating functions analogously to the how the multiplication was expressed in Theorem \ref{mult_gen_fn} and Corollary \ref{multiplication_duality}.
\begin{theorem} \label{hopf_generating_functions}
We have the following equalitiy of generating functions:
\begin{eqnarray*}
& &
\sum_{\vv{\lambda} \in \mathcal{P}^{I}} \left(\prod_{U \in I(\mathcal{C})} s_{\vv{\lambda}(U)}(\underbar{$x$}^{(U)}) \right) \otimes \Delta(\tilde{X}_{\vv{\lambda}}) \\
&=&
\left(\sum_{\vv{\mu} \in \mathcal{P}^{I}} \left(\prod_{U \in I} s_{\vv{\mu}(U)}(\mathbf{x}_U) \right) \otimes (Z_{\vv{\mu}} \otimes 1) \right)
\left(\sum_{\vv{\nu} \in \mathcal{P}^{I}} \left(\prod_{U \in I} s_{\vv{\nu}(U)}(\mathbf{x}_U) \right) \otimes (1 \otimes Z_{\vv{\nu}}) \right)
\end{eqnarray*}
\end{theorem}
\begin{proof}
Recall that $\Delta(T_l(U)) = T_l(U) \otimes 1 + 1 \otimes T_l(U)$ for arbitrary $U$. Therefore: 
\begin{eqnarray*}
& &
\sum_{\vv{\lambda} \in \mathcal{P}^{I}} \left(\prod_{U \in I(\mathcal{C})} s_{\vv{\lambda}(U)}(\mathbf{x}_U) \right) \otimes \Delta(Z_{\vv{\lambda}})  \\
&=&
\prod_{l=1}^{\infty}\exp \left( \Delta \left(T_l \left(\log \left(1 + \sum_{U \in I} p_l(\mathbf{x}_U)U\right)\right) \right) \right)\\
&=&
\prod_{l=1}^{\infty}\exp \left( T_l \left(\log \left(1 + \sum_{U \in I} p_l(\mathbf{x}_U)U\right) \right) \otimes 1 + 1 \otimes T_l \left(\log \left(1 + \sum_{U \in I} p_l(\mathbf{x}_U)U\right) \right) \right) \\
&=&
\prod_{l=1}^{\infty}\exp \left( T_l \left(\log \left(1 + \sum_{U \in I} p_l(\mathbf{x}_U)U\right) \right) \otimes 1 \right)
\prod_{l=1}^{\infty}\exp \left( 1 \otimes T_l \left(\log \left(1 + \sum_{U \in I} p_l(\mathbf{x}_U)U\right) \right) \right) \\
&=&
\left(\sum_{\vv{\mu} \in \mathcal{P}^{I}} \left(\prod_{U \in I} s_{\vv{\mu}(U)}(\mathbf{x}_U) \right) \otimes (Z_{\vv{\mu}} \otimes 1) \right) 
\left(\sum_{\vv{\nu} \in \mathcal{P}^{I}} \left(\prod_{U \in I} s_{\vv{\nu}(U)}(\mathbf{x}_U) \right) \otimes (1 \otimes Z_{\vv{\nu}}) \right) \\
\\
\end{eqnarray*}
\end{proof}
\begin{corollary}\label{coefficient_corollary}
The coefficient of $Z_{\vv{\mu}} \otimes Z_{\vv{\nu}}$ in $\Delta(Z_{\vv{\lambda}})$ is equal to the coefficient of
\[
\prod_{U \in I} s_{\vv{\lambda}(U)}(\mathbf{x}_U)
\]
in 
\[
\prod_{U \in I} s_{\vv{\mu}(U)}(\mathbf{x}_U) \times \prod_{U \in I} s_{\vv{\nu}(U)}(\mathbf{x}_U)
\]
\end{corollary}
\begin{proof}
This immediately follows from Theorem \ref{hopf_generating_functions} by extracting the coefficients of the relevant Schur functions.
\end{proof}
\noindent
We conclude this section by defining some subalgebras that define integral forms of tensor powers of the algebra $U(\mathbb{Q} \otimes_{\mathbb{Z}}R)$.
\begin{definition}
Let $\mathcal{G}_k^{\mathbb{Z}}(R)$ be the subalgebra of $\mathcal{G}_\infty^{\mathbb{Z}}(R)$ generated by $e_i(U)$ with $i \leq k$ and arbitrary $U$.
\end{definition}
\noindent
It easily follows that $\mathcal{G}_k^{\mathbb{Z}}(R)$ is a sub-Hopf algebra of $\mathcal{G}_\infty^{\mathbb{Z}}(R)$. The following proposition describes these subalgebras.
\begin{proposition}
The algebra $\mathcal{G}_1^\mathbb{Z}(R)$ is an integral form of the universal enveloping algebra of $\mathbb{Q} \otimes_\mathbb{Z} R$ (considered as a Lie algebra). This integral form has a $\mathbb{Z}$-basis consisting of PBW monomials in a basis of $R$. Similarly, $\mathcal{G}_k^\mathbb{Z}(R)$ is an integral form of the $k$-fold tensor product of the universal enveloping algebra $U(\mathbb{Q} \otimes_\mathbb{Z} R)$. The Hopf algebra structures coming from $\mathcal{G}_\infty^{\mathbb{Z}}(R)$ and from the universal enveloping algebra coincide. Finally, if $R$ is of finite rank over $\mathbb{Z}$, each $\mathcal{G}_k^\mathbb{Z}(R)$ is noetherian.
\end{proposition}
\begin{proof}
Note that $\mathbb{Q} \otimes_{\mathbb{Z}} \mathcal{G}_k^\mathbb{Z}(R)$ is the subalgebra of $\mathcal{G}_\infty^{\mathbb{Q}}(R)$ generated by the $e_i(U)$ for $i \leq k$. Because $e_n(U)$ is equal to $T_n(U)$ plus a polynomial in $T_m(V)$ with $m<n$ (this follows from extracting the coefficient of $t^n$ from the generating function defining $e_n(U)$), this is the same as the subalgebra generated by the $T_i(U)$ for $i \leq k$. By Theorem 8.8 of \cite{Ryba}, this is the $k$-fold tensor product of the universal enveloping algebra of $\mathbb{Q} \otimes_\mathbb{Z} \mathcal{G}(R)$, and as $T_i(U)$ is primitive, the Hopf algebra structures agree. This proves the statement about $\mathcal{G}_k^\mathbb{Z}(R)$ being an integral form. To show that $\mathcal{G}_k^\mathbb{Z}(R)$ is noetherian, it suffices to consider the associated graded algebra. Note that the $e_i(U)$ for $U \in I$ ($I$ a $\mathbb{Z}$-basis of $R$) are free polynomial generators of the associated graded algebra of $\mathcal{G}_\infty^\mathbb{Z}(R)$ (as per \cite{Nate}). Because Corollary \ref{first_degree_relation} expresses $[e_i(U), e_j(V)]$ in terms of $e_k(W)$ with $k < \min(i, j)$, it follows that the associated graded algebra of $\mathcal{G}_k^\mathbb{Z}(R)$ is a free polynomial algebra generated by $e_i(U)$ for $i \leq k$ and $U \in I$ (which is noetherian, as $I$ was assumed to be finite). It follows that $\mathcal{G}_k^\mathbb{Z}(R)$ is noetherian. 
\end{proof}

\section{Dual of Hopf algebra Structure}
\noindent
In this section we discuss the dual of the Hopf algebra defined in Section 8. Our main result is that $\mathcal{G}_\infty^{\mathbb{Z}}(R)$ is isomorphic to the Hopf algebra of distributions on the formal neighbourhood of the identity in $(W \otimes_{\mathbb{Z}} R)^\times$ (where $W$ is the ring of Big Witt Vectors) that are supported at the identity. In this section, we require that the unit of $R$, denoted $\mathbf{1}$, is an element of $I$.

\begin{definition}
Let $\mathcal{G}_\infty^{\mathbb{Z}}(R)^{*}$ be the (full) dual space of $\mathcal{G}_\infty^{\mathbb{Z}}(R)$. We write $Y_{\vv{\lambda}}$ for the elements of $\mathcal{G}_\infty^{\mathbb{Z}}(R)^{*}$ that are dual to $Z_{\vv{\lambda}}$ (i.e. $Y_{\vv{\lambda}}(Z_{\vv{\mu}}) = \delta_{\vv{\lambda}, \vv{\mu}}$, where we have used the Kronecker delta).
\end{definition}

\begin{proposition} \label{bialgebra_structure}
There is a Hopf algebra structure on $\mathcal{G}_\infty^{\mathbb{Z}}(R)^{*}$ coming from the dual of the multiplication, unit, comultplication and counit on $\mathcal{G}_\infty^{\mathbb{Z}}(R)$ in Theorem \ref{hopf_str}. Additionally, $\mathcal{G}_\infty^{\mathbb{Z}}(R)^*$ is isomorphic to the ring $Q$ of formal power series in $e_r(U)$ ($r \geq 1$, $U \in I$):
\[
Q = \mathbb{Z}[[e_1(\mathbf{1}), e_2(\mathbf{1}), \ldots, e_1(U), e_2(U), \ldots, \ldots]]
\]
where $Y_{\vv{\lambda}} \mapsto \prod_{U \in I} s_{\vv{\lambda}(U)}(\mathbf{x}_U)$.
\end{proposition}

\begin{proof}
Consider the equations in Corollary \ref{coefficient_corollary}. Dualising the comultiplication on $\mathcal{G}_\infty^{\mathbb{Z}}(R)$ gives a multiplication on $\mathcal{G}_\infty^{\mathbb{Z}}(R)^*$, where the coefficient of $Y_{\vv{\lambda}}$ in $Y_{\vv{\mu}}Y_{\vv{\nu}}$ coincides with the coefficient of
\[
\prod_{U \in I} s_{\vv{\lambda}(U)}(\mathbf{x}_U)
\]
in
\[
\prod_{U \in I} s_{\vv{\mu}(U)}(\mathbf{x}_U)\cdot \prod_{U \in I} s_{\vv{\nu}(U)}(\mathbf{x}_U)
\]
Note that this gives a well defined multiplication because for any fixed $\vv{\lambda}$, there are only finitely many pairs $(\vv{\mu}, \vv{\nu})$ for which the coefficient above is nonzero. It immediately follows that $\mathcal{G}_\infty^{\mathbb{Z}}(R)^*$ is isomorphic to $Q$ as an algebra, with the claimed isomorphism. From now on, we prefer to work with the symmetric function realisation. To determine the comultiplication, we write $Q \otimes Q$ with two sets of symmetric function variables: $\{\mathbf{x}_U\}_{U \in I}$ and $\{\mathbf{y}_U\}_{U \in I}$. Then Corollary \ref{multiplication_duality} gives that the comultiplication in $Q$ applied to $\prod_{U \in I} s_{\vv{\lambda}(U)}^{(U)}$ is equal to
\[
\prod_{U \in I} s_{\vv{\lambda}(U)}(\mathbf{x}_U, \mathbf{y}_U, \bigoplus_{V, W \in I}(\mathbf{x}_V\mathbf{y}_W)^{\oplus N_{V,W}^U})
\]
Note that this is well defined because any $\prod_{U \in I} s_{\vv{\mu}(U)}^{(U)}\prod_{U \in I} s_{\vv{\nu}(U)}^{(U)}$ can occur in the image of only finitely many $\prod_{U \in I} s_{\vv{\lambda}(U)}^{(U)}$.
\end{proof}
\noindent
In view of the convenient description using symmetric functions, we prefer to work with $\mathcal{G}_\infty^{\mathbb{Z}}(R)^*$ in terms of symmetric functions (specifically, we view $Q$ as a completion of a tensor product of copies of the ring of symmetric functions) rather than the dual of $\mathcal{G}_\infty^{\mathbb{Z}}(R)$.
\begin{proposition}
The antipode $S$ on $Q$ is implicitly defined by the following equation:
\[
\sum_{U \in I} S(p_l(\mathbf{x}_U)) U = \sum_{r \geq 1}(-1)^r (\sum_{U\in I} p_l(\mathbf{x}_U)U)^r
\]
\end{proposition}
\begin{proof}
Because $T_l(U)$ is primitive, if $S$ is the antipode on $\mathcal{G}_{\infty}^\mathbb{Z}(R)$, $S(T_l(U)) = -T_l(U)$. This implies:
\begin{eqnarray*}
& &
\sum_{\vv{\lambda} \in \mathcal{P}^{I}} \left(\prod_{U \in I} s_{\vv{\lambda}(U)}(\mathbf{x}) \right) \otimes S(Z_{\vv{\lambda}}) \\
&=& \prod_{l=1}^{\infty}\exp \left( S \left(T_l \left(\log \left(1 + \sum_{U \in I} p_l(\mathbf{x}_U U\right)\right) \right) \right) \\
&=& \prod_{l=1}^{\infty}\exp \left( - T_l \left(\log \left(1 + \sum_{U \in I} p_l(\mathbf{x}_U)U\right) \right) \right) \\
&=& \prod_{l=1}^{\infty}\exp \left(T_l \left(\log \left((1 + \sum_{U \in I} p_l(\mathbf{x}_U) U)^{-1}\right)\right) \right) \\
&=& \prod_{l=1}^{\infty}\exp \left(T_l \left(\log \left(1 + 
\sum_{r \geq 1}(-1)^r (\sum_{U \in I} p_l(\mathbf{x}_U)U)^r
\right)\right) \right). \\
\end{eqnarray*}
From this equation, it follows that the dualised antipode (i.e. the antipode on the ring $Q$) on power sum symmetric functions is defined by the following equation (similarly to the proof of Proposition \ref{bialgebra_structure}):
\[
\sum_{U \in I} S(p_l(\mathbf{x}_U)) U = \sum_{r \geq 1}(-1)^r (\sum_{U \in I} p_l(\mathbf{x}_U)U)^r.
\]
\end{proof}
\begin{example}
This sum defining the antipode is not finite. For example, when $R = \mathbb{Z}$, we take $I = \{\mathbf{1}\}$; this gives
\[
S(p_l^{(\mathbf{1})}) = \sum_{r \geq 1} (-1)^r (p_l^{(\mathbf{1})})^{r}.
\]
\end{example}
\begin{proposition} \label{formal_gp_law}
The comultiplication $\Delta: Q \to Q \otimes Q$ defines a formal group law $F$ in the variables $e_i^{(U)}$ ($U \in I$) (essentially by dualisation). The definition is
\[
F(\{e_i(\mathbf{x}_U)\}_{i \geq 1, U \in I}, \{e_i(\mathbf{y}_U)\}_{i \geq 1, U \in I}) = \{e_i(\mathbf{x}_U, \mathbf{y}_U, \bigoplus_{V, W \in I(\mathcal{C})}(\mathbf{x}_V \mathbf{y}_W)^{\oplus N_{V, W}^U})\}_{i \geq 1, U \in I}.
\]
\end{proposition}
\begin{proof}
Coassociativity of the comultiplication implies associativity of the formal group law. The fact that $F$ is addition to first order can be seen by using the fact that $e_i(\mathbf{x}, \mathbf{y}) = \sum_{j=0}^{i} e_j(\mathbf{x})e_{i-j}(\mathbf{y})$:
\[
e_i(\mathbf{x}_U, \mathbf{y}_U, \bigoplus_{V, W \in I}(\mathbf{x}_V\mathbf{y}_W)^{\oplus N_{V, W}^U})
= \sum_{a+b+c=i} e_a(\mathbf{x}_U)e_b(\mathbf{y}_U)e_c(\bigoplus_{V, W \in I}(\mathbf{x}_V\mathbf{y}_W)^{\oplus N_{V, W}^U}).
\]
The only summands that give rise to (a scalar multiple of) an elementary symmetric function are those for which $c=0$ and either $a=0$ or $b=0$. So we get $e_i(\mathbf{x}_U) + e_i(\mathbf{y}_U)$ plus higher order terms, as required.
\end{proof}

\begin{definition}
Let the affine group scheme represented by the commutative ring $\Lambda$ be called the Big Witt Vectors, and denoted $W$. Thus for a commutative ring $A$, the Big Witt Vectors associated to $A$, denoted $W(A)$, are defined as follows. As a set, $W(A) = \hom_{alg}(\Lambda, A)$ is the set of algebra homomorphisms from the ring of symmetric functions to $A$. The addition is induced by the usual comultiplication on $\Lambda$,  $\Delta^{(+)}(e_n) = \sum_{i=0}^{n} e_i \otimes e_{n-i}$. The multiplication is induced by the Kronecker comultiplication $\Delta^{(\times)}(e_n) = \sum_{\lambda \vdash n} s_\lambda \otimes s_{\lambda^\prime}$. 
\end{definition}
\noindent
Note in particular that the underlying additive group of the Big Witt Vectors is the affine group scheme over $\mathbb{Z}$ represented by the ring of symmetric functions $\Lambda$ (with the usual Hopf algebra structure). The addition and multiplication on $W$ satisfy distributivity as discussed in Section 10 of \cite{hazelwinkel}. 
\newline \newline\noindent
Since $\Lambda$ is the free polynomial algebra in the elementary symmetric functions $e_i$ ($i \in \mathbb{Z}_{> 0}$), an element of $W(A)$ is the same thing as a choice of an element of $A$ for each $e_i$. This may be represented as an infinite sequence $(a_1, a_2, \ldots)$, where $a_i \in A$ is the image of $e_i$. Then, the additive identity in $W(A)$ is $(0,0,\ldots)$, and the multiplicative identity is $(1,0,0,\ldots)$ (this is equivalent to Equation 10.24 of \cite{hazelwinkel}, although complete symmetric functions are used rather than elementary symmetric functions).

\begin{theorem} \label{dual_hopf_recognition}
The algebra $Q = \mathcal{G}_\infty^{\mathbb{Z}}(R)^*$ is isomophic to the Hopf algebra of functions on the formal neighbourhood of the multiplicative identity of $(W \otimes_\mathbb{Z} R)^{\times}$.
\end{theorem}
\begin{proof}
Note that the underlying additive group of $W \otimes_\mathbb{Z} R$ is represented (as an affine group scheme) by the ring $Q = \bigotimes_{U \in I} \Lambda^{(U)}$ (with comultiplication coming from the usual comultiplication on $\Lambda$). The multiplication comes from the following comultiplication (obtained by combining $\Delta^{\times}$ with the dual of the multiplication on $R$):
\[
f^{(U)} \mapsto f(\bigoplus_{V, W \in I} (\mathbf{x}_V\mathbf{y}_W)^{\oplus N_{V,W}^U}).
\]
The multiplicative identity in $W$ is given by the sequence $(1,0,0,\ldots)$. This maps $e_1$ to $1$ and all other elementary symmetric functions to zero, giving $(1,0,0,\ldots) \in W$; this is the same as evaluating a symmetric function on the variable set $\{ 1,0,0,\ldots \}$. It then follows that the maximal ideal of $\Lambda$ corresponding to this point is defined by the homomorphism $\varphi: \Lambda \to \mathbb{Z}$ given by $\varphi(s_{(n)}) = 1$ and $\varphi(s_\lambda) = 0$ for $\lambda$ with at least two parts. It also follows that the maximal ideal of $Q$ corresponding to the identity in $W \otimes_{\mathbb{Z}} R$ is the ideal which is the kernel of the homomorphism $\psi: Q \to \mathbb{Z}$ given by $\psi(s_{\vv{\lambda}}) = 1$ when $\vv{\lambda}(\mathbf{1}) = (n)$ for some $n$, and $\vv{\lambda}(U)$ is the empty partition for $U \neq \mathbf{1}$, and $\psi(s_{\vv{\lambda}})=0$ otherwise. Let this maximal ideal be called $J$. Then, we must complete $Q$ with respect to the ideal $J$. To do this, consider the automorphism $\theta$ of $\Lambda^{(\mathbf{1})}$ defined by $\theta(e_i^{(\mathbf{1})}) = e_i^{(\mathbf{1})} - e_{i-1}^{(\mathbf{1})} + e_{i-2}^{(\mathbf{1})} -\cdots$. Note that $\theta(J)$ is the ideal of positive degree elements of $Q$ under the usual grading (it suffices to notice that evaluating $\theta(e_i^{(\mathbf{1})})$ at the variable set $\{1,0,0,\ldots \}$ gives zero for all $i$). Completing with respect to this ideal, we obtain:
\[
\widehat{\bigotimes}_{U \in I} \Lambda^{(U)} = \mathbb{Z}[[e_1^{(\mathbf{1})}, e_2^{(\mathbf{1})}, \ldots, e_1^{(U)}, e_2^{(U)}, \ldots, \ldots]] = Q.
\]
This algebra is isomorphic to $\mathcal{G}_\infty^\mathbb{Z}(R)^*$ by letting $Y_{\vv{\lambda}} \mapsto \prod_{U \in I} s_{\vv{\lambda}(U)}^{(U)}$. To check that this is an isomorphism of Hopf algebras, it remains to check that the comultiplication and antipode agree (here we must take into account the twist by $\theta$). The comultiplication is determined by taking the comultiplication on $Q$ and twisting it by $\theta$. The comultiplication on $Q$ is defined by the following formula (where we interpret $Q \otimes Q$ as being complete symmetric functions in the variables $\mathbf{y}_U$ and $\mathbf{z}_U$ for $U \in I$):
\[
f(\mathbf{x}_U) \mapsto f(\bigoplus_{V, W \in I(\mathcal{C})} (\mathbf{y}_V\mathbf{z}_W)^{\oplus N_{V, W}^U}).
\]
Thus, $\theta$ has the effect of changing the variable sets where either $V$ or $W$ is $\mathbf{1}$. For examples $\mathbf{y}_\mathbf{1}$ becomes $\{1\} \cup \mathbf{y}_{\mathbf{1}}$ (we have appended $1$ to the variable set), and so $\mathbf{y}_{\mathbf{1}}\mathbf{z}_W$ becomes $\mathbf{y}_{\mathbf{1}}\mathbf{z}_W \cup \mathbf{z}_W$. Hence, we obtain a formula for the comultiplication twisted by $\theta$:
\[
f(\mathbf{x}_U) \mapsto f(\mathbf{y}_U, \mathbf{z}_U, \bigoplus_{V,W \in I} (\mathbf{y}_V \mathbf{z}_W)^{\oplus N_{V,W}^U}).
\]
This is in agreement with the formula in Corollary \ref{coefficient_corollary}. Because a bialgebra admits at most one antipode, it automatically follows that the antipodes agree. This completes the proof.
\end{proof}
\begin{remark}
The isomorphism constructed in Theorem \ref{dual_hopf_recognition} depends on $I$.
\end{remark}

\begin{theorem}\label{final_thm}
The algebra $\mathcal{G}_\infty^{\mathbb{Z}}(R)$ is isomorphic to the Hopf algebra of distributions on the formal neighbourhood of the identity in $(W \otimes_\mathbb{Z} R)^\times$ that are supported at the identity.
\end{theorem}
\begin{proof}
It is well known that distributions supported at a point are given by differential operators. Moreover, the Hopf algebra structure comes from the group structure on the formal neighbourhood, and is dual to that of functions on the formal neighbourhood. Thus it suffices to show that differential operators on $Q = \mathcal{G}_\infty^\mathbb{Z}(R)^*$ give rise to $\mathcal{G}_\infty^{\mathbb{Z}}(R)$. We think of $\mathcal{G}_\infty^\mathbb{Z}(R)^*$ as $\mathbb{Z}[[e_1^{(\mathbf{1})}, e_2^{(\mathbf{1})}, \ldots, e_1^{(U)}, e_2^{(U)}, \ldots, \ldots]]$, meaning that differential operators are given by linear combinations of the functionals $f_{\vv{\lambda}}$ which extract the coefficient of $\prod_{U \in I} s_{\vv{\lambda}(U)}^{(U)}$ in an element of $\mathbb{Z}[[e_1^{(\mathbf{1})}, e_2^{(\mathbf{1})}, \ldots, e_1^{(U)}, e_2^{(U)}, \ldots, \ldots]]$. The comultiplication on the $f_{\vv{\lambda}}$ has structure constants equal to those defined by multiplication of the $s_{\vv{\lambda}}$:
\[
\Delta(f_{\vv{\lambda}}) = \sum_{\vv{\mu}, \vv{\nu} \in \mathcal{P}^I} (\prod_{U \in I} c_{\vv{\mu}(U), \vv{\nu}(U)}^{\vv{\lambda}(U)}) f_{\vv{\mu}} \otimes f_{\vv{\nu}}
\]
In particular, when $\vv{\lambda}$ is a basic hook of size $i$ associated to $U \in I$, we obtain the same comultiplication as for $e_i(U)$ in $\mathcal{G}_\infty^{\mathbb{Z}}(R)$. The multiplication is obtained by dualising the comultiplication on $\mathcal{G}_\infty^{\mathbb{Z}}(R)^*$. This is the same as the formal group law in Proposition \ref{formal_gp_law}. It then follows that the multiplications agree. Finally, because bialgebras admit at most one antipode, the equality of antipodes is automatic.
\end{proof}
\begin{remark}
Although we only considered rings $R$ that are free as $\mathbb{Z}$-modules, in view of Theorem \ref{final_thm}, we may extend the definition to arbitrary $R$ in the following way. Define $\mathcal{G}_\infty^{\mathbb{Z}}(R)$ as the Hopf algebra of distributions on the formal neighbourhood of the identity in $(W \otimes_\mathbb{Z} R)^\times$ that are supported at the identity. This definition is functorial in $R$.
\end{remark}
\bibliographystyle{alpha}
\bibliography{wreath_alg.bib}

\end{document}